\documentclass[11pt]{article}
\usepackage{hyperref}
\usepackage{color}
\usepackage{amsmath,amsthm,amssymb,amsfonts}
\usepackage{enumerate}
\usepackage{graphicx}
\usepackage{bm}
\usepackage{dsfont}
\def\R{\mathbb R}
\def\N{\mathbb N}

\def\e{\varepsilon}
\def\trait (#1) (#2) (#3){\vrule width #1pt height #2pt depth #3pt}
\def\fin{\hfill\trait (0.1) (5) (0) \trait (5) (0.1) (0) \kern-5pt
\trait (5) (5) (-4.9) \trait (0.1) (5) (0)}

\numberwithin{equation}{section}

\newcommand{\be}{\begin{equation}}
\newcommand{\ee}{\end{equation}}
\newcommand{\baa}{\begin{array}}
\newcommand{\eaa}{\end{array}}
\newcommand{\ba}{\begin{eqnarray}}
\newcommand{\ea}{\end{eqnarray}}

\newcommand{\ban}{\begin{eqnarray*}}
\newcommand{\ean}{\end{eqnarray*}}

\newtheorem{theo}{\bf Theorem}[section]
\newtheorem{lem}[theo]{\bf Lemma}
\newtheorem{pro}[theo]{\bf Proposition}

\newtheorem{defi}[theo]{\bf Definition}

\theoremstyle{remark}

\newtheorem{rem}[theo]{\bf Remark}

\newtheorem{example}[theo]{\bf Example}
\def\ds{\rightarrow}
\def\e{\varepsilon}

\def\Ac{{\cal A}}
\def\Tc{{\cal T}}

\def\Up{\mathbf{U}^+}
\def\Um{\mathbf{U}^-}

\def\apg{\left\{}
\def\chg{\right\}}
\def\apt{\left(}
\def\cht{\right)}

\def\ch{\right.}

\def\a{a}

\def\A{{\cal A}}

\def\HTreg{{ H}^{\rm reg}_T}

\def\AHregx{A_\H^{\rm reg}(x)}

\def\eps{\varepsilon}
\renewcommand{\H}{\mathcal{H}}

\newcommand{\mB}{\mathcal{B}}
\newcommand{\cob}{\overline{\mathop{\rm co}}}

\def\xe{x_\eps}

\def\HT{H_T}

\def\xe{x_{\gamma}}

\def\xne{(x_{\gamma})_N}

\def\ye{y_{\gamma}}
\def\ype{y'_{\gamma}}
\def\yne{(y_{\gamma})_N}

\newcommand{\T}{\mathcal{\Im}}
\newcommand{\Uim}{\mathbf{U}^\mathrm{FL}}

\newcommand{\ind}[1]{\mathds{1}_{{#1}}}
\newcommand{\mT}{\mathcal{T}}
\renewcommand{\d}{\,\mathrm{d}}

\def\ds{\rightarrow}
\newcommand{\Ii}{\mathcal{I}_{i}}

\newcommand{\I}{\mathcal{I}}

\newcommand{\ueta}{u_\eta}
\newcommand{\bx}{\bar x}

\newcommand{\Hti}{\tilde{H}}
\newcommand{\Htireg}{\tilde H^{\rm reg}}

\newcommand{\Honemin}{\underline{H}_1}
\newcommand{\Htwomin}{\underline{H}_2}

\begin{document}

\title{Flux-limited and classical viscosity solutions for regional control problems}

\author{G.Barles, A. Briani, E. Chasseigne\thanks{Laboratoire de
    Math\'ematiques et Physique Th\'eorique (UMR CNRS 7350),
    F\'ed\'eration Denis Poisson (FR CNRS 2964), Universit\'e
    Fran\c{c}ois Rabelais, Parc de Grandmont, 37200 Tours,
    France. Email: Guy.Barles@lmpt.univ-tours.fr,
    Ariela.Briani@lmpt.univ-tours.fr,
    Emmanuel.Chasseigne@lmpt.univ-tours.fr.  } , C. Imbert\thanks{CNRS
    \& D\'epartement de mathématiques et applications (UMR CNRS 8553),
    \'Ecole Normale Sup\'erieure (Paris), 45 rue d'Ulm, 75230 Paris
    cedex 5, France. Email: Cyril.Imbert@ens.fr\newline
    \indent This work was partially supported by the ANR project HJnet (ANR-12-BS01-0008-01)} 
}
\maketitle

\begin{abstract} The aim of this paper is to compare two different
  approaches for regional control problems: the first one is the
  classical approach, using a standard notion of viscosity solutions,
  which is developed in a series of works by the three first
  authors. The second one is more recent and relies on
  ideas introduced by Monneau and the fourth author for problems set
  on networks in another series of works, in particular
  the notion of flux-limited solutions. After describing and even
  revisiting these two very different points of view in the simplest
  possible framework, we show how the results of the classical
  approach can be interpreted in terms of flux-limited solutions. In
  particular, we give much simpler proofs of three results: the
  comparison principle in the class of bounded flux-limited solutions
  of stationary multidimensional Hamilton-Jacobi equations and
 the
  identification of the maximal and minimal Ishii's solutions with
  flux-limited solutions which were already proved by Monneau and the
  fourth author, and the identification of the corresponding vanishing
  viscosity limit, already obtained by Vinh Duc Nguyen and the fourth
  author.
\end{abstract}

%-------------------------------------
\section{Introduction}
%-------------------------------------

Recently, a lot of works have been devoted to the study of
deterministic control problems involving discontinuities and, more
precisely, problems where the dynamics and running costs may be
completely different in different parts of the domain. In fact, these
problems can be of different natures: first, they may only deal with
``simple'' discontinuities of codimension $1$ like in \cite{DeZS},
\cite{GS1,GS2}, \cite{So}; the first three authors provide in
\cite{BBC1,BBC2} a systematic study of such problems and we describe
these results below. Second, following Bressan \& Hong \cite{BrYu},
other results are concerned with problems in ``stratified domains'',
where the discontinuities can be of any codimension; we refer to
\cite{AEYW} for a new and simpler approach of these problems, with new
results.  Third, they are problems set on networks for which the
specified methods are required since such singular domains are not
necessarily contained in $\R^N$; we refer to \cite{ACCT}, \cite{IMZ},
\cite{ScCa}, \cite{IM}, \cite{IM-md}, \cite{IN} \cite{LiSo}, for different
approaches of such networks problems.

The aim of this article is to compare the different approaches used in
these articles, and in particular the ones of \cite{BBC1,BBC2} and
\cite{IM,IM-md}. Indeed, this link is only presented
  in the mono-dimensional setting in \cite{IM}; see also
  \cite{IM-md}.  In order to provide the clearest possible picture,
we consider the simplest possible case, namely the case of two
half-spaces in $\R^N$, say $\Omega_1:=\{x=(x_1,\cdots,x_N); x_N>0\}$
and $\Omega_2:=\{x=(x_1,\cdots,x_N); x_N<0\}$ and we also choose below
the most simple assumptions on either the control problem or the
Hamilton-Jacobi Equations (controllability or coercivity). In the same
line, we restrict ourselves to the case of stationary Hamilton-Jacobi
equations, corresponding to infinite-time horizon control problems
(with actualization factor $\lambda=1$).

The first key step, and this is one major difference in the above
mentioned works, is to identify the questions we are interested in
and/or the methods we are able to use. This is where the fact to be in
$\R^N$ or on a network changes completely the point of view. In
\cite{BBC1,BBC2}, the key questions were the following. First,
consider the equations
\begin{equation}\label{HJ1}
u+H_1(x,Du) = 0 \quad\hbox{in  }\Omega_1\; ,
\end{equation}
\begin{equation}\label{HJ2}
u+H_2(x,Du) = 0 \quad\hbox{in  }\Omega_2\; ,
\end{equation}
then the classical Ishii's definition of viscosity solutions implies that 
we have ``natural junction conditions'' on $\H:= \overline{\Omega}_1 \cap
\overline{\Omega}_2 = \apg  x \in \R^N  \: : \: x_N=0 \chg$ which read
\begin{equation}\label{HJ-subH}
\min(u+H_1(x,Du),u+H_2(x,Du)) \leq 0 \quad\hbox{on  }\H\; ,
\end{equation}
\begin{equation}\label{HJ-superH}
\max(u+H_1(x,Du),u+H_2(x,Du)) \geq 0 \quad\hbox{on  }\H\; .
\end{equation}
Indeed, if $H$ is the Hamiltonian defined by
$$ H(x,u,p):= \left\{\begin{array}{ll}u+H_1(x,p) & \hbox{if $x\in \Omega_1$}\\
u+H_2(x,p) & \hbox{if $x\in \Omega_2$}\end{array}\right.$$
then the above inequalities are nothing but $H_*\leq 0$ and $H^* \geq 0$ on $\H$.
Unfortunately, these junction conditions are not enough to ensure 
 uniqueness and there may (and in general do) exist several Ishii's discontinuous solutions.

The first question which is addressed in \cite{BBC1,BBC2} is to define
properly a control problem where the dynamics and running cost are
different in $\Omega_1$ and $\Omega_2$. The main problem concerns the
controlled trajectories which may stay on $\H$: how to properly define
them and do they lead to the junction conditions
\eqref{HJ-subH}-\eqref{HJ-superH}?  Then the next question is to
identify the maximal and minimal solutions of
\eqref{HJ1}-\eqref{HJ2}-\eqref{HJ-subH}-\eqref{HJ-superH} when $H_1$,
$H_2$ are Hamiltonians of control problems (see
Theorem~\ref{thm:Um.Up} at the end of Section~\ref{sect:BBC}). A key
remark on these results is that the use of differential inclusions
methods leads on $\H$ to a mixing of the dynamics-costs of $\Omega_1$
and $\Omega_2$ and this is actually (depending on the type of mixing
one allows) how the maximal and minimal solutions of
\eqref{HJ1}-\eqref{HJ2}-\eqref{HJ-subH}-\eqref{HJ-superH} are defined.
This approach (refered below as CVS $=$ classical viscosity solutions'
approach) is described in Section~\ref{sect:BBC} with the main
results.

In the network framework, the question of how to define the junction
condition(s) becomes more central since the definition of
  classical Ishii's definition of viscosity solutions is not
  straightforward in the general case.  Such a difficulty is related
  to another important difference (which is not addressed at all in
\cite{BBC1,BBC2}) which is the choice of the set of test-functions:
while in $\R^N$, even with the discontinuities on $\H$, the choice of
test-functions which are $C^1$ in $\R^N$ is natural, this choice makes
no sense in the network framework where the ``natural'' set of
test-functions is the set of functions which are $C^1$ on
each branch and continuous at the junctions. Here, if test-functions
are chosen to be continuous in $\R^N$, $C^1$ in $\Omega_1$ and
$\Omega_2$ and to have a trace on $\H$ which is $C^1$ on $\H$
(allowing a jump on the $x_N$-derivative), the question is: what does
this change in the \cite{BBC1,BBC2} picture?

In order to answer this question, we first describe the flux-limited
solution approach (FL-approach in short) consisting in
adding a junction condition $G$ on $\H$. It can be seen as being
associated to a particular control problem on $\H$. This function $G$
is called the \emph{flux limiter} in \cite{IM,IM-md}. Compared to
\cite{BBC1,BBC2}, this approach is more PDE-oriented: we give and
comment the definition with test-functions which are just piecewise
$C^1$. Even if it is rather natural from the control point of view, it
turns out to be rather different from the classical Ishii's
definition.

For the FL-approach, we provide a simplified uniqueness proof for the associated
  Hamilton-Jacobi-Bellman Equations obtained in \cite{IM-md}. Instead
  of using the so-called vertex test function (which construction is
  difficult and lengthy), we simply use specific slopes identified in
  \cite{IM,IM-md} (see Lemma~\ref{linkHami} in Appendix) in order to
  construct a simple test function. Indeed, it is explained in
  \cite{IM,IM-md} that a function is a flux-limited solution if it
  satisfies the viscosity inequality on $\H$ only when tested with
  smooth functions whose derivatives at the junction coincide with
  those specific slopes. We do not need such a result about the
  reduction of test functions here but, guided by this idea, we give a
  simpler proof of the comparison principle.  Finally we identify the
value-function ($\Uim_G$) which is the unique solution of this problem
associated to $G$.

The next question is the comparison of the two (apparently very
different) approaches in the multi-dimensional setting: it turns out
that, as in the mono-dimensional setting \cite{IM},
the maximal ($\Up$) and minimal ($\Um$) solutions in the CVS-approach
can be recovered by using the right ``flux limiter'' $G$ (or control
problem) on $\H$: these flux limiters are respectively the
Hamiltonians $H_T^{\rm reg}$ and $H_T$ identified in
\cite{BBC1,BBC2}. We conclude that the FL-approach provides a
completely different way (and with pure PDE methods) to
address the questions solved in \cite{BBC1,BBC2}. Moreover, the choice
of $G$ (in particular the case when there is no such a flux limiter)
allows one to consider different control problems on $\H$ in a more
general way than in \cite{BBC1,BBC2}.

Last but not least, this clear understanding on the advantages and
disadvantages of the two points of view for looking at the HJ problem
with discontinuities, allows us to simplify the proof of the
convergence of the vanishing viscosity approximation, a result already
given in \cite{IN}.

The article is organized as follows: in Section~\ref{IM}, we describe
the FL-approach with the simplified comparison proof and the
connection with the related control problem. Then in
Section~\ref{sect:BBC}, we recall the CVS-approach; the two approaches
are compared in Section~\ref{sect:comp.IM.BBC}.  The convergence of
the vanishing viscosity approximation closes the article
(Section~\ref{sect:vanishing}). The appendix contains technical
results which are used in the paper.

%------------------------------------
\section{Flux-limited solutions}\label{IM}
%------------------------------------

%--------------------------------
\subsection{Assumptions and definitions}

We first describe the assumptions on the dynamic and running cost in
each $\Omega_i\ (i=1,2)$ and on $\H$ since they are used to define the
junction conditions. We recall that we use the simplest possible
assumptions and we formulate the problem in the simplest possible way
by assuming that the dynamics and running costs are defined in the
whole space $\R^N$.

On $\Omega_i$, the sets of controls are denoted by $A_i$, the system
is driven by a dynamic $b_i$ and the running cost is given by
$l_i$. We use the index $i=0$ for $\H$.  Our main assumptions are the
following.

\begin{itemize}

\item[{[H0]}] For $i=0,1,2$, $A_i$ is a compact metric space and
  $b_i: \R^N \times A_i \ds \R^N$ is a continuous bounded function,
  more precisely $|b_i(x,\alpha_i)| \leq M_{b}$ for all $x \in \R^N$
  and $\alpha_i \in A_i$, $i=0,1,2$.  Moreover, there exists
  $L_i \in \R$ such that, for any $x,y \in \R^N$ and
  $\alpha_i \in A_i$
$$ |b_i(x,\alpha_i)-b_i(y,\alpha_i)|\leq L_i |x-y|\; .$$

\item[{[H1]}] For $i=0,1,2$, the function
  $l_i: \R^N \times A_i \ds \R^N$ is continuous and
  $|l_i(x,\alpha_i)| \leq M_l$ for all $x \in \R^N$ and
  $\alpha_i \in A_i$, $i=1,2$.
\end{itemize}
\noindent The last assumption is a controlability assumption that we
use only in $\Omega_1\cup\Omega_2$, and not on $\H$.
\begin{itemize}
\item[{[H2]}] For each $x \in \R^N$, the sets
  $\apg (b_i(x, \alpha_i),l_i(x, \alpha_i)) \:: \: \alpha_i \in A_i
  \chg$,
  ($i=1,2$), are closed and convex. Moreover there is a $\delta>0$
  such that for any $i=1,2$ and $x\in\R^N$,
\begin{equation}\label{cont-ass}
\overline{B(0,\delta)} \subset B_i(x):= \apg   b_i(x, \alpha_i)  \: :     \:   \alpha_i \in A_i \chg.
\end{equation}

\end{itemize}

We now define several Hamiltonians. For $x \in \overline{\Omega}_1$
\be  \label{def:Ham1}
H_1(x,p):=\sup_{ \alpha_1 \in A_1} \apg  -b_1(x,\alpha_1) \cdot p  - l_1(x,\alpha_1) \chg\,,
\ee
\be  \label{def:Ham1-}
H_1^-(x,p):=\sup_{ \alpha_1 \in A_1\:: \: b_1(x,\alpha_1) \cdot e_N  \leq  0} \apg  -b_1(x,\alpha_1) \cdot p  - l_1(x,\alpha_1) \chg\,,
\ee
\be  \label{def:Ham1+}
H_1^+(x,p):=\sup_{ \alpha_1 \in A_1\:: \: b_1(x,\alpha_1) \cdot e_N > 0} \apg  -b_1(x,\alpha_1) \cdot p  - l_1(x,\alpha_1) \chg\,,
\ee
and for $x \in \overline{\Omega}_2$
\be  \label{def:Ham2}
H_2(x,p):=\sup_{ \alpha_2 \in A_2} \apg  -b_2(x,\alpha_2) \cdot p  - l_2(x,\alpha_2) \chg\,,
\ee
\be  \label{def:Ham2+}
H^+_2(x,p):=\sup_{ \alpha_2 \in A_2  \:: \: b_2(x,\alpha_2) \cdot e_N \geq 0} \apg  -b_2(x,\alpha_2) \cdot p  - l_2(x,\alpha_2) \chg\,,
\ee
\be  \label{def:Ham2-}
H^-_2(x,p):=\sup_{ \alpha_2 \in A_2  \:: \: b_2(x,\alpha_2) \cdot e_N < 0} \apg  -b_2(x,\alpha_2) \cdot p - l_2(x,\alpha_2) \chg\,.
\ee

Finally, for the specific control problem on $\H$ we define for any $x\in\H$ and
$p_\H\in\R^{N-1}$
\be\label{def:G}
G(x,p_\H):=\sup_{\alpha_0\in\A_0}\{-b_0(x,\alpha_0)\cdot p_\H-l_0(x,\alpha_0)\}\,.
\ee

In the sequel, the points of $\H$ are identified indifferently by
$x'\in\R^{N-1}$ or by $x=(x',0)\in\R^N$.  For the gradient variable we
use the decomposition $p=(p_\H,p_N)$ where $p_\H\in \H= \R^{N-1}$ and
$p_N\in \R$, and, when dealing with a function $u$, we also use the
notation $D_\H u$ for the ${(N-1)}$ first components of the gradient,
i.e.,
\[ 
D_\H u:=(\frac{\partial u}{\partial x_1}, \cdots, \frac{\partial
  u}{\partial x_{n-1}})\quad\hbox{and}\quad Du=\Big(D_\H
u,\frac{\partial u}{\partial x_N}\Big)\,.
\]
Note that, for the sake of consistency of notation, we also denote by
$D_\H u$ the gradient of a function $u$ which is only defined on
$\R^{N-1}$.

Let us remark that, thanks to assumptions [H0], [H1], the Hamiltonians $H_i$,
$H_i^{\pm}$ ($i=1,2$) satisfy the following classical
  structure conditions: for any $R>0$, for any $x,y\in \R^N$ such that $|x|,|y| \leq R$, for any $p,q \in \R^N$ and for $i=1,2$
  \begin{equation}\label{eq:structure}
 \begin{cases} 
|H_i (x,p)-H_i(x,q)| \le M_b |p-q| \\
|H_i(x,p)-H_i (y,p)| \le L_i |x-y|(1+|p|)+m_i^R(|x-y|)\; ,
\end{cases}
\end{equation}
where $m_i^R$ is  a (non-decreasing) modulus of continuity of the function $l_i$ on the compact set $\overline{B(0,R)}\times A_i$.

The assumptions on the function $G$ mimic the assumptions naturally satisfied by $H_1, H_2$.
\begin{itemize}
\item[{[HG]}] The function $G: \H \times \R^{N-1} \rightarrow \R$ is
  continuous and satisfies: for any $x \in \H$, the function
  $p'\mapsto G(x,p'): \R^{N-1} \rightarrow \R$ is convex and there
  exist $C_1 , C_2 >0$ and, for any $R$, a modulus of continuity $m_R^G$ such that, for any $x,y\in \H$ with $|x|,|y| \leq R$,
for any $ p^\prime \in \R^{N-1}$
\[
|G(x, p^\prime )-G(y, p^\prime )| \leq C_1 |x-y|(|p^\prime| +1) m_R^G(|x-y|) \quad
,\quad |G (x,p^\prime)-G(x,q^\prime)| \leq C_2|p^\prime-q^\prime| \: .
\]
\end{itemize} 
 We point out that, because of Lemma~\ref{ss-lip} below, the
coercivity of $G$ is not necessary.

We introduce the following space $\T$ of real valued test-functions:
we say that $\psi \in \T$ if $\psi \in C(\R^N)$ and these exist
$\psi_1 \in C^1(\bar \Omega_1)$, $\psi_2 \in C^1(\bar \Omega_2)$ such
that $\psi = \psi_1$ in $\bar \Omega_1$ and $\psi = \psi_2$ in
$\bar \Omega_2$. Of course, $\psi_1=\psi_2$ and
$D_\H \psi_1= D_\H \psi_2$ on $\H$.

Now we give a definition of sub and supersolution following
\cite{IM,IM-md} for the following problem 
\[ \tag{HJ-FL} 
\begin{cases}
u+H_1(x,Du) = 0 \quad\hbox{in  }\Omega_1\; , \\
u+H_2(x,Du) = 0 \quad\hbox{in  }\Omega_2\; , \\
u+G(x,D_\H u) = 0 \quad\hbox{on  }\H\; .
\end{cases} 
\]
Since in $\Omega_1, \Omega_2$, the definition are just classical
viscosity sub and supersolutions, we only provide the definition on
$\H$.
%----------------------------------------------------------------
\begin{defi}[Flux-limited sub and supersolution on $\H$]
  \label{defiNEW} An upper semi-continuous (usc), bounded function
  $u: \R^N \rightarrow \R$ is a \emph{flux-limited subsolution of
    (HJ-FL) on $\H$} if for any test-function $\psi \in \T$ and any
  local maximum point $x \in \H$ of $x \mapsto (u-\psi)(x)$ in $\R^N$,
  we have
\[
  \max \Big( u(x)+G(x,D_\H \psi) ,  u(x)+ H_1^+(x, D\psi_1) ,  
  u(x)+ H_2^-(x, D\psi_2) \Big) \leq 0  \: .
\]
We say that a lower semi-continuous (lsc), bounded function
$v: \R^N \rightarrow \R$ is a \emph{flux-limited supersolution of
  (HJ-FL) on $\H$} if for any function $\psi \in \T$ and any local
mininum point $x \in \H$ of $x \mapsto (v-\psi)(x)$ in $\R^N$, we have
\[
  \max \Big(v(x)+G(x,D_\H \psi) ,   v(x)+ H_1^+(x, D\psi_1) ,  
  v(x)+ H_2^-(x, D\psi_2) \Big) \geq 0  \: .
\]
\end{defi}
%--------------------------------------------------------------------
\begin{rem}
Let us point out that, in Definition~\ref{defiNEW}, the local extrema are taken
with respect to a neighborhood of $x$ in $\R^N$ and not with respect to a
neighborhood of $x$ in $\H$ as in \cite{BBC1, BBC2, AEYW}. This definition is
``natural'' in the sense that it takes into account dynamics $b_1$ pointing
inward to $\Omega_1$ in $H_1^+$ and in the same way dynamics $b_2$ pointing
inward to $\Omega_2$ in $H_2^-$. This is also why flux-limited subsolutions can exist
since with test-functions in $\T$ and a natural extension of the Ishii's
definition using $\psi_1$ in $H_1$ and $\psi_2$ in $H_2$, we would have no
subsolutions (consider $x \mapsto u(x)-|x|^2/\e^2 -C_\e |x_N|$, for a large
constant $C_\e$). But it can also be noticed that a subsolution of
$u+H_1(x,Du)=0$ in $\Omega_1$ satisfies naturally $u+H_1^+(x,Du) \leq 0$ on
$\H$, the same being true with $H_2$, $\Omega_2$ and $H_2^-$ (see \cite{BBC1}).
\end{rem}

%------------------------------------
\subsection{Comparison result for flux-limited sub/supersolutions}

The first natural result we provide is the
%-------------------------------------------------------------------------
\begin{lem}[Subsolutions are Lipschitz continuous] 
\label{ss-lip} Assume [H0]-[H2] and [HG]. Any bounded, usc
    flux-limited subsolution of (HJ-FL) is Lipschitz continuous. 
\end{lem}
%-------------------------------------------------------------------------
\begin{rem}
  In the case of equations of evolution type, or equivalently in the
  case of finite horizon control problems, subsolutions are no longer
  Lipschitz continuous (not even in the space variable). But the
  regularization arguments of \cite{BBC1,BBC2}, using sup-convolution
  in the ``tangent'' variable together with a controlability
  assumption in the normal variable, allows one to reduce to the case
  when the subsolution is Lipschitz continuous (and even $C^1$ in the
  tangent variable if the Hamiltonians are convex).
\end{rem}
We skip the proof of Lemma~\ref{ss-lip} since it follows the classical
PDE proof (see \cite[Lemma 2.5, p. 33]{Ba}) using that $H_1, H_2$ and
$\max( G , H_1^+, H_2^-)$ are coercive function in $p$ (uniformly in
$x$); we notice that $\max(H^+_1,H^-_2)$  is a coercive
function in $p$ --- see  Remark~\ref{lemCOE} in Appendix for the case
of $\max(H^-_1,H^+_2)$, which is equivalent.

The main result of this section is the following. 
 %---------------------------------------------------------------------
\begin{theo}[Comparison principle] \label{comp-IM} Assume [H0]-[H2]
  and [HG]. If $u,v :\R^\N \ds \R$ are respectively a usc bounded
  flux-limited subsolution and a lsc bounded flux-limited
  supersolution of (HJ-FL) then $u \leq v$ in~$\R^N$.
\end{theo}
\begin{rem}
  This result is proved in the evolution setting in \cite{IM-md}.  But
  the proof presented below is much simpler, avoiding in particular
  the use of the vertex test function.
\end{rem}
%%%%%%%%%%%%%%%%%%%%%%%%%%
\begin{proof} The first step of the proof consists in {\em localizing}
  as in \cite[Lemma 4.3]{BBC1}: for $K>0$ large enough, the function
  $\psi:=-K -(1+|x|^2)^{1/2}$ is a classical flux-limited subsolution
  of (HJ-FL). For $\mu \in ]0,1[$ close to $1$, the function
  $u_\mu:=\mu u + (1-\mu) \psi $ is also Lipschitz continuous
  (cf. Lemma~\ref{ss-lip}) and an flux-limited subsolution of (HJ-FL)
  by using the convexity of $H_1, H_2, G$.  Moreover
  $u_\mu(x) \to -\infty$ as $|x|\to +\infty$.

    The proof consists in showing that, for any $\mu\in (0,1)$,
    $u_\mu \leq v$ in $\R^N$ and then in letting $\mu$ tend to~$1$ to
    get the desired result.  Since $u_\mu (x)-v(x) \to -\infty$ as
    $|x|\to +\infty$, there exists $\bar{x} \in \R^N$ such that
$$
 M:=u_\mu (\bar{x})-v(\bar{x})=\sup_{x \in \R^N} \:  \big( u_\mu (x) - v(x) \big)  \:.
 $$
We assume by contradiction that $M >0$.

We first remark that, necessarily, $\bar{x} \in \H$. Indeed, otherwise
we can use classical comparison arguments for the $H_1$ or $H_2$
equation, together with an easy localisation argument, to get a
contradiction.

Next we consider a first doubling of variables by introducing the map 
\[
 (x',y',x_N) \mapsto u_\mu (x,x_N) - v(y',x_N)-\frac{|x'-y'|^2}{\e^2} \; .
\]
Using again the (negative) coercivity of $u_\mu$, this function reaches its
maximum $M_\eps$ at $(\tilde x',\tilde y',\tilde x_N)$ and this point is a global strict
maximum point of 
\[ 
(x',y',x_N) \mapsto u_\mu (x',x_N) -
v(y',x_N)-\frac{|x'-y'|^2}{\e^2} -|x'-\tilde x'|^2- |y'-\tilde y'|^2-|x_N-\tilde
x_N|^2\; .
\] 
Since we have $M = \lim_{\eps \to 0} M_\eps$, we can choose $\eps \in (0,\eps_0)$ so that 
$M_\eps \ge M/2 >0$.

\noindent {\bf CASE A:} 
$\tilde x_N>0$ or $\tilde x_N<0$. We introduce a new parameter $0<\gamma \ll 1$
and the function $$ (x,y) \mapsto u_\mu (x',x_N) -
v(y',y_N)-\frac{|x'-y'|^2}{\e^2}-\frac{|x_N-y_N|^2 }{\gamma^2} -|x'-\tilde
x'|^2- |y'-\tilde y'|^2-|x_N-\tilde x_N|^2\; .$$ 
Since we have $M_\eps = \lim_{\gamma \to 0} M_{\eps,\gamma}$, we can choose
$\gamma \in (0,\gamma_0)$ so that $M_{\eps,\gamma} \ge M/4>0$.

We are going to explain below that in Case A the conclusion follows
easily using the coercivity of $H_1$ or $H_2$, but with a little
modification from the standard case.

Assume for instance that $\tilde x_N>0$. Since the maximum points
$x=(x',x_N)$ and $y=(y',y_N)$ of this function respectively converge
to $(\tilde x',\tilde x_N)$ and $(\tilde y',\tilde x_N)$ when
$\gamma \to 0$, we conclude that $x,y\in\Omega_1$ for $\gamma$ small
enough.  Using the sub and supersolution conditions with Hamiltonian
$H_1$ we get
$$
\begin{aligned}
    & u_\mu(x',x_N)+H_1((x',x_N),D_x\psi_1)\leq0\\
    & v(y',y_N)+H_1((y',y_N),-D_y\psi_1)\geq0\,,
\end{aligned}
$$
where 
$$\psi_1(x,y)=\frac{|x'-y'|^2}{\e^2}+\frac{|x_N-y_N|^2 }{\gamma^2} +|x'-\tilde
x'|^2 + |y'-\tilde y'|^2+|x_N-\tilde x_N|^2\,.
$$
 The coercivity of $H_1$ (or the fact that subsolutions are Lipschitz continuous) implies by the subsolution condition 
that $|D_x\psi_1(x',x_N)|\leq C$ for some $C>0$ independent of $\eps,\gamma>0$.
In particular 
\begin{equation}\label{est-pg}
\frac{2|x_N-y_N|}{\gamma^2} \leq C\footnote{We point out here that if we were assuming normal controlability instead of complete controlability, this property would be replaced by
$$\frac{2|x_N-y_N|
}{\gamma^2} \leq C \left(\frac{2| x'- y'|}{\e^2} +1\right)\; ,$$
and the whole argument would still work.}.
\end{equation}

Subtractring the sub/supersolution conditions and using the standard
structure properties [H0] and [H1] of $H_1$ (see \eqref{eq:structure})
we get
\[
\begin{aligned}
u_\mu(x',x_N)-v(y',y_N)& \leq
m\bigg( |x'-y'|\Big(1+2\frac{|x'-y'|}{\eps^2}+2\frac{|x_N-y_N|}{\gamma^2}
+2|y'-\tilde y'|\Big)\bigg)\\
& + C\Big( 2|y'-\tilde y'|+2|x'-\tilde x'|+2|x_N-\tilde x_N|\Big)\\
& \leq
m\bigg( (1+C)|x'-y'|+2(1+C)\frac{|x'-y'|^2}{\eps^2}
+2|x'-y'| \; |y'-\tilde y'|\Big)\bigg)\\
& + C\Big( 2|y'-\tilde y'|+2|x'-\tilde x'|+2|x_N-\tilde x_N|\Big)
\end{aligned}
\]
for some (non-decreasing) modulus of continuity $m(\cdot)$ (we used
\eqref{est-pg}).  We let first $\gamma\to0$ and then $\eps\to0$. Then,
we end up with the usual contradiction: $M\leq0$. Of course, if
$\tilde x_N<0$ we use the $H_2$ sub/supersolution conditions for
$u_\mu$ and $v$.

\noindent {\bf CASE B:} $\tilde x_N = 0$. We set
$\tilde p':= \displaystyle \frac{2(\tilde x'-\tilde y')}{\e^2}$ and 
$$A:=-\Big(\frac{u_\mu (\tilde x',0) + v(\tilde x',0)}{2}\Big)\,.$$
Notice that by our choice, $-u_\mu(\tilde x',0)<A<-v(\tilde x',0)$.

To proceed, we are going to use the following lemma
  whose proof is postponed until the end of the proof of
  Theorem~\ref{comp-IM}.
  \begin{lem}\label{boncas} When $\tilde x_N=0$, we have
$$ u_\mu (\tilde x',0) + \max(\Honemin ({\tilde x},
\tilde p'),\Htwomin ({\tilde x},\tilde p')) \leq 0\; .$$
\end{lem}

Since, by Lemma~\ref{boncas},
$-u_\mu(\tilde x',0) \geq \max(\Honemin (\tilde x,\tilde p'),\Htwomin
(\tilde x,\tilde p')) $,
the inequality
\[ \max(\Honemin (z,\tilde p'),\Htwomin (z,\tilde p')) < A \]
still hold, for $\eps >0$ small enough, where
$$z=\Big(\frac{\tilde x'+\tilde y'}{2}\,,\,0\Big)\,.$$
Indeed, for such $\eps$, $A\geq -u_\mu(\tilde x',0) + M/2$, while
\[ \max(\Honemin (\tilde x,\tilde p'),\Htwomin (\tilde x,\tilde p'))\quad  \text{ is close to
} \quad \max(\Honemin (z,\tilde p'),\Htwomin (z,\tilde p')).\]

Hence, by Lemma~\ref{linkHami} in the Appendix, there exist a unique pair 
$\lambda_2 < \lambda_1$, solution of 
$$
H_{1}^- (z,\tilde p'+\lambda_1 e_N)=
A\quad,\quad  H_{2}^+ (z,\tilde p'+\lambda_2 e_N)=
A \; .$$

 In order to build the test-function, we set
 $h(t):=\lambda_1t_+-\lambda_2t_-$ (with $t_+= \max(t,0)$ and
 $t_-=\max(-t,0)$) and
\begin{equation} \label{defchila}
\chi(x_N,y_N): = h(x_N)-h(y_N)=
\apg
\begin{array}{ll}
    \lambda_1 ( x_N - y_N) & \mbox{ if }  x_N \geq 0 \: , \: y_N \geq 0\,, \\
    \lambda_1  x_N-\lambda_2
    y_N & \mbox{ if }  x_N \geq 0 \: , \: y_N < 0\,, \\
    \lambda_2 x_N - \lambda_1
    y_N &  \mbox{ if }  x_N < 0 \: , \: y_N \geq 0\,, \\
    \lambda_2( x_N - y_N ) & \mbox{ if }  x_N < 0 \: , \: y_N < 0 \,.
\end{array} \ch
\end{equation}
Now, for $0<\gamma\ll \e$ we define a test function as follows
$$
\psi_{\e,\gamma}(x,y):= \frac{|x'-y'|^2}{\e^2} + \chi(x_N,y_N)+
\frac{|x_N-y_N|^2 }{\gamma^2}+|x'-\tilde x'|^2+|y'-\tilde
y'|^2+|x_N-\tilde x_N|^2\; .$$
In view of the definition of $h$, we see that for any $x \in \R^N$ the
function $\psi_{\e,\gamma}(x,\cdot) \in \T$ and for any $y \in \R^N$
the function $\psi_{\e,\gamma}(\cdot,y) \in \T$.

Dropping the $\eps$-reference but keeping the $\gamma$ one, 
let us define $\xe=(x'_{\gamma};(x_{\gamma})_N)$ and 
$\ye=(y'_{\gamma};(y_{\gamma})_N)$, the maximum points of
$u_\mu(x)-v(y)-\psi_{\e,\gamma}(x,y)$. More
precisely  
\[
u_\mu(\xe)-v(\ye)-\psi_{\e,\gamma}(\xe,\ye)= \max_{(x,y) \in
    \R^N \times \R^N} \apt   u_\mu(x)-v(y)-\psi_{\e,\gamma}(x,y) \cht  \: .
\]
Because of the localisation terms, we have, as $\gamma \to 0$,
$ x_\gamma \to (\tilde x',0)$ and $y_\gamma \to (\tilde y',0)$. From
now on, we are going to drop the localisation terms to simplify the
expressions, keeping just their effects which are all of $o(1)$ types.

We have to consider different cases depending on the position of $\xe$
and $\ye$ in $\R^N$. Of course, using again the coercivity of $H_1$ or
$H_2$, we have no difficulty for the cases $\xne,\yne >0$ or
$\xne,\yne < 0$; only the cases where $\xe$, $\ye$ are in different
domains or on $\H$ cause problem. For the sake of simplicity of
notation, write $\psi$ for $\psi_{\e,\gamma}$ and
$(\lambda_1,\lambda_2)$ where actually those
parameters depend on $\eps,\gamma$.

  For the sake of
clarity we start by summarizing the arguments we use to get a
contradiction for the various subcases.
\begin{itemize}
\item
Subcases B-(a) and B-(b): we use the subsolution condition for
$u_\mu$ and $u_\mu+A>0$ .
\item Subcases B-(c) and B-(d): we use the supersolution for $v$ and
$v+A<0$ .
\item
Subcase B-(e): we use the FL-definition on the interface.
\end{itemize}

Now we detail the proofs.

\noindent\textbf{Subcase B-(a):} $\xne >0$, $\yne \leq 0$.

Let us assume first that $\yne <0$. Since $\xe \in \Omega_1$ therefore
we look at $\xe$ as a local  maximum point in $\Omega_1$ of the
function
\[
 x \mapsto u_\mu(x)-v(\ye)- \frac{|x'-\ype|^2}{\e^2}-(\lambda_1 x_N-\lambda_2 \yne) 
-  \frac{|x_N-\yne|^2 }{\gamma^2} + \text{(localization terms)}.
\] 
Since $u_\mu$ is a subsolution of $u_\mu(x)+H_1(x,Du_\mu) = 0$   in  $ \Omega_1$, this implies that 
\begin{equation}\label{ineq:subsol}
u_\mu(\xe)+  H_1(\xe,D_x\psi(\xe,\ye)) \leq 0 
\end{equation}
where 
\[
D_x\psi(\xe,\ye)=p'_{\gamma} + \lambda_1  e_N + 2  \frac{\xne-\yne}{\gamma^2}e_N + o(1)
\: ,
\]
with $\displaystyle p'_{\gamma} = 2  \frac{(\xe-\ye)}{\e^2}$.
We point out that $p'_{\gamma} \to \tilde p'$ as $\gamma \to 0$ and
therefore $p'_\gamma= \tilde p'+o_\gamma(1)$.

Notice first that since $u_\mu$ is Lipschitz continuous, $D_x\psi$ is
bounded and by [H0]-[H1] (analogously to \eqref{eq:structure}) there
exists a modulus of continuity $\omega(\cdot)$ (independent of
$\gamma$ and $\e)$ such that
$$
|H_{1}^-(\xe,D_x\psi(\xe,\ye))- H_{1}^- (z,
D_x\psi(\xe,\ye))|\leq \omega(|\xe-z|)\,.
$$
Since $\xe \to (\tilde x',0)$ and since $|z-(\tilde x',0)|=o_\e(1)$, we have $|\xe-z|=o_\gamma(1)+o_\e(1)$.
Then, using also the monotonicity of $H^-_1$ in the $p_N$-variable
(see Lemma \ref{lem:H1m.H2p} in the Appendix) we have 
$$
H_{1}^-(\xe,D_x\psi(\xe,\ye)) \geq  H_{1}^- (z, \tilde p' +
\lambda_1  e_N) +o_\gamma(1)+o_\e(1) \; .
$$
Then we use that $H_1\geq H^-_{1}$ and since $ u_\mu(\xe)=u_\mu(\tilde
x',0)+o_\gamma(1)$, we get, using the definition of~$\lambda_1$
$$\begin{aligned}
0 \geq u_\mu(\xe)+  H_1(\xe,D_x\psi(\xe,\ye)) \geq &\  
u_\mu(\tilde
x',0)+  H_{1}^- (z, \tilde p' +
\lambda_1  e_N) +o_\gamma(1)+o_\e(1)\\ 
\geq &u_\mu(\tilde
x',0)+  A +o_\gamma(1)+o_\e(1)\,.
 \end{aligned}
$$
But $u_\mu(\tilde x',0)+A>0$, therefore if $\gamma\ll \eps $ are small enough, we get a contradiction with
\eqref{ineq:subsol} since $M>0$.
Finally, the same argument works for $\yne=0$, 
changing the $y_N$-term in $\chi$.

\noindent\textbf{Subcase B-(b):} $\xne <0$, $\yne \geq 0$.

Since the argument is symmetrical to the first case, we omit the proof: we just
use the subsolution condition with $H^+_{2}$ and the definition of $\lambda_2$
instead of $H^-_{1}$ and the definition of $\lambda_1$.

\noindent\textbf{Subcase B-(c):} $\xne =0$, $\yne > 0$.

On the one hand, since $\xe \in \H$ the
FL-definition yields 
\[\begin{aligned}
  \max \Big(  u_\mu(\xe)+ G(\xe, D_\H\psi(\xe,\ye))\ ;\  &
  u_\mu(\xe)+ H_1^+(\xe, D_x\psi_1(\xe,\ye))\  ;\ \\
  & u_\mu(\xe)+ H_2^-(\xe, D_x\psi_2(\xe,\ye)) \Big) \leq 0
  \end{aligned}
\]
which implies in particular 
\begin{equation}\label{ineq:case3.sub}
u_\mu(\xe)+ H_1^+(\xe, D_x\psi_1(\xe,\ye))  \leq 0
\end{equation}
where 
\[D_x \psi_1 (\xe,\ye) = p'_\gamma + \lambda_1 e_N - \frac{2}{\gamma^2} (y_\gamma)_N  e_N + o(1).\]
On the other hand, since $v$ is a supersolution of $v+H_1(y,Dv)=0$ in
$\Omega_1$ this implies \be \label{dis1} v(\ye)+ H_1(\ye, - D_y
\psi_1(\xe,\ye)) 
 \geq 0 \ee where
\[
D_y \psi_1(\xe,\ye)= -p'_\gamma -\lambda_1  e_N +  \frac{ 2 }{\gamma^2}
\:   \yne e_N + o(1)\,.
\]
Our goal is to show that the above viscosity inequality holds with
$H_{1}^+$ instead of $H_1$. Indeed, combined with
\eqref{ineq:case3.sub}, this implies $u_\mu (\xe) \le v(\ye) + o(1)$;
passing to the limit in $\gamma$ and $\eps$ respectively, we reach
the contradiction $M = u_\mu (\bar x) - v (\bar x)\le 0$.

In order to do so, since
$H_1=\max(H_{1}^-,H_{1}^+)$ it is enough to show that
\[ 
v(\ye)+ H_{1}^- (\ye, - D_y \psi_1(\xe,\ye))<0\,.
\]
We use similar arguments as in case 1: first, the gap between
$H^-_{1}$ taken at $\xe$ and $\ye$ is controlled by a modulus of
continuity $\omega$.  Then, since $2\yne/\gamma^2>0$ we can use the
monotonicity property of $H^-_{1}$ which gives \be \label{dis2}
v(\ye)+ H_{1}^- (\ye, - D_y \psi_1(\xe,\ye)) \leq v(\ye) + H_{1}^- (z,
p'_\gamma + \lambda_1 e_N ) +o_\gamma(1)\,.  \ee Recalling that
$v(\ye) \to v(\tilde x',0)$, even if $v$ is just lower
semi-continuous, and using the definition of $\lambda_1$ we see that
$$
v(\ye)+ H_{1}^- (\ye, - D_y \psi_1(\xe,\ye)) \leq v(\tilde y',0) +A +o_\gamma(1)+o_\e(1)$$
But $v(\tilde y',0) +A <0$ and if $\gamma\ll \e$ are small enough we get the desired strict inequality.
Therefore, for $\gamma\ll \e$  small enough, we have necessarily 
\begin{equation}\label{ineq:case3.sup}
v(\ye)+H_1^+ (\ye, - D_y \psi_1(\xe,\ye)) \geq 0 \:.
\end{equation}
The conclusion follows by combining \eqref{ineq:case3.sup} and
\eqref{ineq:case3.sub}, and letting first $\gamma$ tend to $0$, then $\e$.

\noindent\textbf{Subcase B-(d):} $\xne =0$, $\yne < 0$. 

The proof is symmetrical to case 3 above: the FL-condition gives a
subsolution condition for $H^-_{2}$ and the supersolution condition is
obtained by using $H^+_{2}$ (instead of $H^-_{1}$ as in the previous
case).

\noindent\textbf{Subcase B-(e):} $\xne =0$, $\yne =0$.

In this case we have both $\xe$ and $\ye$ in $\H$ therefore we have to
use the fact that $u_\mu$ and $v$ are respectively a flux-limited
subsolution and a flux-limited supersolution. Applying carefully
Definition~\ref{defiNEW}, we have
$$
  \max \Big( u_\mu (\xe)+ G(\xe, p'_\gamma)\,;\, u_\mu(\xe)+ H_1^+(\xe, p'_\gamma +
  \lambda_1e_N)\,;\,  u_\mu(\xe)+ H_2^-(\xe, p'_\gamma + \lambda_2e_N) \Big) \leq 0
  \: .
$$
$$
  \max \Big( v(\ye)+ G(\ye, p'_\gamma)\,;\, v(\ye)+ H_1^+(\ye, p'_\gamma +
  \lambda_1e_N)\,;\, v(\ye)+ H_2^-(\ye, p'_\gamma + \lambda_2e_N) \Big) \geq 0  \: .
$$
And the conclusion follows again by letting successively $\gamma$ and $\eps$ tend to $0$.
\end{proof}

\begin{proof}[Proof of Lemma~\ref{boncas}.]
We recall that $(\tilde x',\tilde y',0)$ is a global strict maximum point of 
\[ 
(x',y',x_N) \mapsto u_\mu (x',x_N) -
v(y',x_N)-\frac{|x'-y'|^2}{\e^2} -|x'-\tilde x'|^2- |y'-\tilde y'|^2-|x_N|^2\; .
\] 
In particular, $\tilde x'$ is a global strict maximum point of 
$$x' \mapsto u_\mu (x',0) -
v(\tilde y',0)-\frac{|x'-\tilde y'|^2}{\e^2} -|x'-\tilde x'|^2\; .$$
And we introduce the function
$$(x',x_N) \mapsto u_\mu (x',x_N) -
v(\tilde y',0)-\frac{|x'-\tilde y'|^2}{\e^2} -|x'-\tilde x'|^2-L|x_N|\; ,$$
where $L>0$ is a large constant. 

Choosing $L=L(\eps)$ large enough, the maximum of this new function is
necessarily reached for $x_N=0$: indeed, if $x_N>0$ or $x_N<0$, the
viscosity subsolution inequalities cannot hold because of the
coercivity of $H_1$ and $H_2$.

Therefore this maximum is achieved at $\tilde x=(\tilde x',0)$ and Definition~\ref{defiNEW}, we have
$$
  \max \Big( u_\mu (\tilde x)+ G(x, \tilde p')\,;\, u_\mu(\tilde x)+ H_1^+(\tilde x, \tilde p' +
  L.e_N)\,;\,  u_\mu(\tilde x)+ H_2^-(\tilde x, \tilde p' -L.e_N) \Big) \leq 0
  \: .
$$
In particular, according to the definition of $\Honemin (\tilde x,\tilde p'),\Htwomin (x,\tilde p')$
$$ \max \Big(u_\mu(\tilde x)+ \Honemin (x,\tilde p')\,;\,  u_\mu(\tilde x)+ \Htwomin (x,\tilde p') \Big) \leq 0\; ,$$
which gives the desired inequality.
\end{proof}

\begin{rem}[Extension to second order equations]
The (simplified) proof of Theorem~\ref{comp-IM} can be generalized to treat the case of second-order equations, provided that the junction condition remains first-order; this means that (\ref{HJ1})-(\ref{HJ2}) can be replaced by
$$u+H_i(x,Du)-{\rm Tr}(a_i(x)D^2u) = 0 \quad\hbox{in  }\Omega_i\; ,$$
where the $a_i$'s satisfy : for $i=1,2$, there exist $N\times p$, Lipschitz continuous matrices $\sigma_i$ such that $a_i=\sigma_i.\sigma_i^T$, $\sigma_i^T$ being the transpose matrix of $\sigma_i$, with $\sigma_i((x',0))=0$ for all $x'\in \R^{N-1}$. 

Then, Case~A ($\tilde x_N\neq 0$) follows from classical "second-order" proof,
doubling doubling variables with only one parameter $\eps$, 
both for $x'$ and $\tilde x_N$. 
For Case~B, let us only notice that the second-order terms generated by our
penalizations are either small as $x_\gamma$ and/or $y_\gamma$ 
approaches the interface (because $\sigma_i$ for $i=1,2$ vanishes there and is
Lipschitz continuous), or they simply do not exist if we are on the
interface since the equation degenerates to a first-order one. Hence the proofs
apply as such.
\end{rem}

\subsection{Link with control problems}
\label{sect:control.im}

In order to describe the control problem, we first have to define the
admissible trajectories. We say that $X(\cdot)$ is an admissible
trajectory if
\begin{enumerate}[\rm (i)]
\item there exists a global control $a=(\alpha_1,\alpha_2,\alpha_0)$
  with $\alpha_i\in \mathcal{A}_i:=L^\infty(0,\infty;A_i)$ for
  $i=0,1,2$,
\item there exists a partition $\I=(\I_1, \I_2, \I_0)$ of
  $(0,+\infty)$, where $\I_1, \I_2, \I_0$ are measurable sets, such
  that $X(t)\in\overline{{\Omega}_i}$ for any $t\in \I_i$ if $i=1,2$ and
  $X(t)\in\H$ if $t\in \I_0$,
\item $X$ is a Lipschitz continuous function such that, for almost every $t >0$ 
\be
\dot{X}(t)=b_1(X(t),\alpha_1(t))\ind{\I_1}(t)+b_2(X(t),\alpha_2(t))
\ind{\I_2}(t)+b_0(X(t),\alpha_0(t))\ind{\I_0}(t)\,.  
\ee
\end{enumerate}
The set of all admissible trajectories $(X,\I,a)$ issued from a point $X(0)=x\in\R^N$ is
denoted by $\mT_x$. Notice that under the controllability assumption of $b_1$
and $b_2$, for any point $x\in\R^N$ the constant trajectory $X(t)=x$ 
is admissible so that $\mT_x$ is never void. 

The value function (with actualization factor $\lambda=1$) 
is then defined as 
\begin{multline*}
    \Uim_G(x):=\inf_{(X,\I,a)\in\mT_x}  \int_0^{+\infty} 
    \bigg\{  l_1(X(t),\alpha_1(t))\ind{\I_1}(t)+
        l_2(X(t),\alpha_2(t))\ind{\I_2}(t)\\
         + l_0(X(t),\alpha_0(t))\ind{\I_0}(t)\bigg\}
        e^{-t}\d t
\end{multline*}
where $(l_0,l_1,l_2)$ are running costs defined in $\H,\Omega_1,\Omega_2$
respectively.

By standard arguments based on the Dynamic Programming Principle and
the above comparison result, we have the
%-------------------------------------------------------------------
\begin{theo} The value function $\Uim_G$ is the unique FL-solution of
(HJ-FL). 
\end{theo}
%----------------------------------------------------------------------
\begin{rem}
   In \cite{IM}, deriving the Hamilton-Jacobi equation in
  the finite horizon case is more difficult. Indeed, taking into
  account trajectories which oscillate around the junction point (Zeno
  phenomenon) induce some technical difficulties. 
\end{rem}
\begin{rem}
It is worth pointing out that, in this approach, the partition in
$\I_1, \I_2, \I_0$ implies that there is no mixing on $\H$ between the
dynamics and costs in $\Omega_1$ and $\Omega_2$, contrarily to the BBC
approach (see below). A priori, on $\H$, we have an independent
control problem and no interaction between $(b_1,l_1)$ and
$(b_2,l_2)$.
\end{rem}
\begin{rem}
Partially connected to the previous remark, here we cannot solve the
controlled differential equation by the differential inclusion tools because once
given the sets $\I=(\I_1,\I_2,\I_0)$, the associated set-valued map defining
the dynamics and costs need not be upper semicontinuous. Indeed, in general 
$b_0$ need not be related to the $(b_i)_{i=1..2}$, except for special choices
of $G$ --- see Section \ref{sect:comp.IM.BBC}. 
\end{rem}

%-------------------------------------------
\section{The regional control problem}
\label{sect:BBC}
%-------------------------------------------

We describe now the optimal control problem related to the
Hamilton-Jacobi equation studied in \cite{BBC1,BBC2}. It is referred
to as the \emph{regional control problem}. The basic framework remains the
same as for the FL framework, assumptions [H0]-[H1]-[H2] being
exactlty the same. We keep the same notation when no difference
arises between the two frameworks.

The difference concerns the controlled dynamics and trajectories 
which may stay for a while on the common boundary $\H$: instead of [HG],  
here the dynamics on $\H$ are naturally induced by convex combinations of the
dynamics in $\overline{\Omega}_1$ and $\overline{\Omega}_2$.  More precisely, if $z\in\H$
we set
\begin{equation} \label{fH}
    b_\H\big(z,a)=b_\H\big(z,(\alpha_1, \alpha_2, \mu)\big):=\mu
    b_1 (z,\alpha_1) + (1-\mu)b_2(z,\alpha_2)\,,
\end{equation}
where $\mu \in [0,1]$, $\alpha_1 \in A_1$, $\alpha_2 \in A_2$.  For any
$z\in\H$ and we denote here by
\begin{equation*} 
    A_\H(z):=\Big\{\a=(\alpha_1, \alpha_2, \mu):
	b_\H\big(z,(\alpha_1, \alpha_2, \mu)\big)\cdot
	  e_N(z) = 0\Big\}\,,
\end{equation*}
and the associated cost on $\H$ is
\begin{equation}  \label{lH}
    l_\H(z,\a)=l_\H\big(z,(\alpha_1, \alpha_2, \mu)\big):=\mu
    l_1 (z,\alpha_1) + (1-\mu)l_2(z,\alpha_2)\,.
\end{equation}

Here, the trajectories can be defined by using the approach
through differential inclusions: a trajectory
$X(\cdot)$ issued from $x\in\R^N$ is a Lipschitz continuous functions solution
of the following differential inclusion
\begin{equation} \label{def:traj} 
  \dot X (t) \in \mB(X (t)) \quad \hbox{for a.e.  } t \in
  [0,\infty)  \: ; \quad X(0)=x 
\end{equation} 
where 
\begin{equation}
  \mB(z):= \begin{cases} \mathbf{B}_i(z)   &   \text{ if }
      z\in\Omega_i\,,    \\ \cob \big( \mathbf{B}_1(z) \cup
      \mathbf{B}_2(z) \big)  &  \text{ if } z\in\H\,, 
   \end{cases} 
\end{equation}
the notation $\cob(E)$ referring to the convex closure of the set
$E\subset\R^N$.  
As we see,  controls $a(\cdot)$ can take two forms: either $a(s)$
belongs to one of the control sets $A_i$; or it can be expressed as a
triple $(\alpha_1,\alpha_2,\mu)\in A_1\times A_2\times[0,1]$.  Hence, in
order to define globally a control, we introduce the compact set 
    $A:= A_1\times A_2\times [0,1]$
and define a control as
being a function of  $\Ac:= L^\infty (\R^+; A)$.  From the differential
inclusion we also recover the sets 
\begin{equation*}
    \Ii:= \big\{ t\in \R^+:X (t) \in \Omega_i  \big\} \,,\quad
    \I_\H:= \big\{t\in \R^+:X (t) \in \H   \big\} \,,
\end{equation*}
and the trajectories are then precisely described in the following theorem 
from \cite{BBC1}. 
%----------------------------------------------------------------------------
\begin{theo}[{\cite[Theorem 2.1]{BBC1}}]\label{def:dyn}
 Assume [H0], [H1] and [H2]. Then
 \begin{enumerate}[\rm (i)]
     \item For each $x\in \R^N$,  there exists a
	 Lipschitz function $X: \R^+ \to \R^N$ which is a
	 solution of the differential inclusion  \eqref{def:traj}.
     \item For each solution  $X(\cdot)$ of
	 \eqref{def:traj},  there exists a control $\a(\cdot)\in \Ac$
	 such that for a.e. $ t \in \R^+$
         \begin{equation}
	     \begin{aligned}\label{fond:traj} 
		 \dot X (t) &= \sum_{i=1,2}b_i\big(X
		 (t),\alpha_i(t)\big)\mathds{1}_{\Ii}(t)+
		 b_\H\big(X,a(t)\big)\mathds{1}_{\I_\H}(t)
             \end{aligned}
        \end{equation} 
	where $a(t)=\big(\alpha_1(t),\alpha_2(t),\mu(t)\big)$ if
	$X(t)\in\H$.  
    \item We have
    	\begin{equation*}
		b_\H\big(X(t),a(t)\big)\cdot  e_N \big(X(t)) = 0 \quad \hbox{for a.e.
		}t\in\I_\H\;.  
	\end{equation*}
	In other words, $a(t)\in A_\H(X(t))$ for a.e.
	$t\in\I_\H$.
  \end{enumerate} 
\end{theo}
%----------------------------------------------------------------------------
As in Section \ref{sect:control.im} we introduce the set $\Tc_{x}$ of admissible
controlled trajectories starting from $x$, as the set of $(X,a)$ such that $X$
is Lipschitz, $X(0)=x$ and $(X,a)$ and satisfies \eqref{fond:traj}. This set is
not void because we can solve it as above, by differential inclusion. We now
introduce two kind of strategies on $\H$.

Given $z\in\H$, we call \textit{singular} a dynamic
$b_\H(z,\a)$ with $a=(\alpha_1, \alpha_2, \mu)\in A_\H(z)$ when 
\begin{equation*}
    b_1(z,\alpha_1)\cdot e_N(z) > 0\,, \quad
    b_2(z,\alpha_2)\cdot  e_N (z)< 0\,.  
\end{equation*}
Conversely, the \textit{regular} dynamics are
those for which  $ b_1(z,\alpha_1)\cdot e_N(z) \leq0$ and $  b_2(z,\alpha_2)\cdot  e_N (z)\geq  0 $.
Then, the regular trajectories are defined as 
 \begin{equation*}
     \Tc_{x}^{\rm reg}:= \Big\{ (X,\a) \in   \Tc_{x}:
	  \text{ for a.e. } t \in\I_\H,  \:
	  b_\H\big(X(t),a(t)\big) \mbox{ is regular}
	  \Big\}\,.  
 \end{equation*}
The cost associated to $(X,\a)  \in  \Tc_{x}$
is similar to the one in Section \ref{sect:control.im}, where $l_\H$ is given
by~\eqref{lH}: 
\begin{equation*} 
    \ell(X,a):= \sum_{i=1,2}
	l_i\big(X(t),\alpha_i(t)\big) \mathds{1}_{\Ii}(t) +
        l_\H\big(X(t),\a(t)\big)\mathds{1}_{\I_\H}(t)\,, 
\end{equation*}
however, here we define to value functions according to whether we minimize
the cost on $\Tc$ or $\Tc^{\rm reg}$:
for each $x\in\R^N$ we set
\begin{equation} 
    \Um (x):= \inf_{(X,\a) \in  \Tc_{x} } \int_0^{+\infty} \ell(X, \a)e^{-t}\d
    t\,,\quad
    \Up (x):=\inf_{(X,\a) \in  \Tc_{x}^{\rm reg} } \int_0^{+\infty}
    \ell(X, \a)e^{-t}\d t\,.
\end{equation}
 
Under assumptions [H0]-[H1]-[H2], $\Um$ and $\Up$ fulfill a classical Dynamic
Programming Principle, are bounded and Lipschitz continuous from $\R^\N$
into $\R$ (see \cite[Theorem 2.2, Theorem 2.3]{BBC1}).   
>From the pde viewpoint, in each set $\Omega_i$ both $\Um$ and $\Up$ satisfy the
Hamilton-Jacobi equation $H_i(x,u,Du)=0$ where the $H_i$ are defined by
\eqref{def:Ham1} and \eqref{def:Ham2}.
Now, in order to describe what is happening on the hypersurface $\H$, we
introduce two "tangential Hamiltonians", namely $\HT,\HTreg$.

Recall that if $\phi\in C^1(\H)$, and
$x\in\H$, we denote by $D_\H \phi(x)$ the gradient of $\phi$ at
$x$, which belongs to the tangent space of $\H$ at $x$, identified with
$\R^{N-1}$. 
The Hamiltonian $\HT(x,p_\H)$ is defined for
$(x,p_\H)\in \H\times \R^{N-1}$ as follows:
\begin{equation}\label{def:HamHT} 
    \HT(x,p):=\sup_{A_\H(x)} \big\{  -b_\H(x,\a)\cdot p_\H  -  l_\H(x,\a)   \big\}
\end{equation} where $A_\H(x)$ has been already defined above  and 
\begin{equation}\label{def:HamHTreg}
    \HTreg(x,p):=\sup_{\AHregx} \big\{-b_\H(x,\a)\cdot p_\H-l_\H(x,\a)\big\} 
\end{equation} where for $x\in\H$, 
$$
    \AHregx:=\big\{\a=(\alpha_1,\alpha_2,\mu)  \in A_\H(x)\: ; \:
     b_1(z,\alpha_1)\cdot e_N(z) \leq0  \mbox{ and }   b_2(z,\alpha_2)\cdot  e_N (z)\geq  0 
    	\big\}\,.
$$
\begin{rem} \label{remHpm}
Note that in $\HTreg$ we are considering the controls  as  in the definitions of  $H^-_1$ and $H^+_2$, 
\eqref{def:Ham1-}- \eqref{def:Ham2+},  see also Lemma \ref{linkHami} for further consequences. 
\end{rem} 
The definition of viscosity sub and super-solutions for $\HT$ and
$\HTreg$ have to be understood on $\H$ as follows: 
\begin{defi}[Viscosity subsolutions in $\H$]   \label{defi:sousolH}
    A bounded usc function $u:\H \to\R$ is a \emph{viscosity
    subsolution} of
$$
	u(x)+\HT(x,D_\H u)=0\quad \text{on}\quad
	\H
$$
if, for any $\phi\in C^1(\H)$ and any maximum point
    $x$ of $z \mapsto u(z)-\phi(z)$ in $\H$,
    one has
    \begin{equation*}
	\phi(x)+\HT\big(x,D_\H\phi(x)\big)\leq0\;.
    \end{equation*}
\end{defi}

A similar definition holds for $\HTreg$, for supersolutions and
solutions.  The result proved in \cite{BBC1} is the following.
%-----------------------------------------------------------------------------
\begin{theo}[{\cite[Theorem 2.5 and Corollary 4.4]{BBC1}}] \label{thm:Um.Up}
    Assume [H0], [H1] and [H2]. Then
\begin{enumerate}[\rm (i)]
\item
The value function $\Um$ is   the unique viscosity  solution of 
 \begin{equation}\label{eqn-minus}
 \begin{cases}
   u +H_1(x,Du) = 0     &   \hbox{ in   }\Omega_1 \,, \\
     u +H_2(x,Du) = 0     &   \hbox{ in   }\Omega_2\,, \\
       \min\{ u+ H_1(x,Du),  u+H_2(x,Du)\}\leq0  &\hbox{on   } \H \;, \\
       \max\{ u + H_1(x,Du), u+H_2(x,Du)\}\geq 0 &  \hbox{on   } \H   
\end{cases} 
\end{equation}
fulfilling 
$$ u(x)+\HT(x,D_\H u)\leq 0\quad \text{on}\quad \H, $$
in the sense of Definition~\ref{defi:sousolH}.
\item
 Moreover $\Um$ is the minimal supersolution and solution of
\eqref{eqn-minus} and $\Up$ is the maximal subsolution  and solution of
\eqref{eqn-minus}.
\end{enumerate}
\end{theo}
%--------------------------------------------------------------------

%------------------------------------------------------
\section{Value functions of regional control are flux-limited solutions}
\label{sect:comp.IM.BBC}
%------------------------------------------------------

We recall that $\Uim$ is the value function of the Imbert-Monneau
control problem when there is no ``flux limiter'' $G$, while $\Uim_G$
stands for this value function when $G$ is the flux limiter. The main
result of this section is the following.
%------------------------------------------------------------------------------------
\begin{theo}\label{compBBC-IM}
Under the assumptions of Theorem~\ref{comp-IM} (comparison result), we have
\begin{enumerate}[\rm (i)]
\item $\Um \leq \Up \leq \Uim$ in $\R^N$.
\item $\Um = \Uim_G$ in $\R^N$ if $G=H_T$.
\item $\Up = \Uim_G$ in $\R^N$ if $G=H_T^{\rm reg}$.
\end{enumerate}
\end{theo}
%----------------------------------------------------------------------------------
\begin{rem}
  This result is proved in \cite{IM} in the monodimensional
  setting. In \cite[Proposition~4.1]{IM-md}, it is proved in the
  multidimensional setting that $\Um$ and $\Up$ are flux-limited
  solutions but it is not proved that the corresponding flux functions
  are precisely $H_T$ and $H_T^{\rm reg}$. The fact that
    the flux function corresponding to $U^+$ is $H_T^{\rm reg}$ is
    proved in \cite{IN}.
\end{rem}
\begin{proof} For $(i)$, the inequalities can just be seen as a consequence of the
    definition of $\Um ,\Up , \Uim$ remarking that we have a larger set of
    dynamics-costs for $\Um$ and $\Up$ than for $\Uim$. From a more pde point
    of view, applying   \cite[Lemma 5.3, p.115]{Ba}, it is easy to see that
    $\Um ,\Up$ are flux-limited subsolutions of (HJ-FL) since they are subsolutions of $$
    u(x)+ H_1^+(x, Du) \leq 0 \quad\hbox{in  }\Omega_1\; ,$$
$$ u(x)+ H_2^-(x, Du) \leq 0 \quad\hbox{in  }\Omega_2\; .$$
Then Theorem~\ref{comp-IM} allows us to conclude.

For $(ii)$ and $(iii)$, we have to prove respectively that $\Um$ is a solution
of (HJ-FL) with $G=H_T$ and $\Up$ with $G=H_T^{\rm reg}$. Then the equality is just
a consequence of Theorem~\ref{comp-IM}.

For $\Um$, the subsolution property just comes from the above argument
for the $H_1^+, H_2^-$-inequalities and from \cite{BBC1} (Theorem~2.4)
for the $H_T$-one. The supersolution inequality is a consequence of
the ``magic lemma'' (Theorem~3.3 in \cite{BBC1}): alternative {\bf A)}
implies that one of the $H_1^+, H_2^-$-inequalities hold while
alternative {\bf B)} implies that the $H_T$-one holds.

For $\Up$, the subsolution property follows from the same arguments as
for $\Um$, both for the $H_1^+, H_2^-$-inequalities and from
\cite{BBC1} (Theorem~2.4) for the $H_T^{\rm reg}$-one. The
supersolution inequality is a consequence of the ``particular magic
lemma'' for $\Up$ (Theorem~2.5 in \cite{BBC1}): alternative {\bf A)}
implies that one of the $H_1^+, H_2^-$-inequalities hold while
alternative {\bf B)} implies that the $\HTreg$-one holds.

And the proof is complete.
\end{proof}

Inequalities in Theorem~\ref{compBBC-IM}-$(i)$ can be strict: various
examples are given in \cite{BBC1}. The following one in dimension $1$
shows that we can have $\Up < \Uim$ in $\R$.
%----------------------------------------------------------------------
\begin{example}  Let
    $\Omega_1=(0,+\infty)$, $\Omega_2=(-\infty, 0)$. We choose 
\[
    b_1(\alpha_1)=\alpha_1 \in [-1,1]\; , \; l_1(\alpha_1)= \alpha_1\; ,
\]
\[
 b_2(\alpha_2)=\alpha_2 \in [-1,1]\; , \; l_1(\alpha_2)= -\alpha_2\; .
\]
It is clear that the best strategy is to use $\alpha_1=-1$ in
$\Omega_1$, $\alpha_2=1$ in $\Omega_2$ and an easy computation gives
\[ \Up(x) = \int_0^{+\infty} -\exp(-t)dt = -1\; ,\]
because we can use these strategies in $\Omega_1$, $\Omega_2$ but also at $0$ since the combination 
\[ \frac12  b_1(\alpha_1) +\frac12  b_2(\alpha_2) = 0\; ,\]
has a cost $-1$. In other words, the ``push-push'' strategy at $0$ allows to maintain the $-1$ cost.

But for $\Uim$, this ``push-push'' strategy at $0$ is not allowed and,
since the optimal trajectories are necessarely monotone, the best
strategy when starting at $0$ is to stay at $0$ but here with a best
cost which is $0$. Hence $\Uim(0) = 0 > \Up(0)$ and it is easy to show
that $\Uim(x) > \Up(x)$ for all $x\in \R$.
\end{example}
%----------------------------------------------------------------

Theorem~\ref{compBBC-IM} can be interpreted in several ways: first the main
information is that (of course) the key point is what kind of controlled
trajectories we wish to allow on $\H$ and, depending on this choice, different
formulations have to be used for the associated HJB problem. It could be thought
that the flux-limited approach is more appropriate, in particular because of
Theorem~\ref{comp-IM} which is used intensively in the above proof. 

%So {\bf ``Peace and love nous sommes tous freres''!}

%----------------------------------
\section{Vanishing viscosity approximation}\label{sect:vanishing}
%----------------------------------

We begin this section with a general remark on the stability
properties of both types of solutions. On the one hand, classical
viscosity solutions are defined in such a way that they are stable
(under half relaxed limits) and this is one of their main
advantages. On the other hand, {\em in our framework}, they are not
unique, i.e. there are in general several classical viscosity
solutions lying between the minimal one $\Um$ and the maximal one
$\Up$. On the contrary, flux-limited solutions are unique but their
stability under half relaxed limits is less straightforward: we refer
to \cite{IM,IM-md} for the proof that flux-limited
solutions are stable.

The vanishing viscosity method provides us with an example
  where this difference is clear: with Ishii's definition, one can
  pass to the (semi-)limit(s) and obtain
  \eqref{HJ1}-\eqref{HJ2}-\eqref{HJ-subH}-\eqref{HJ-superH} in a
  standard way and it immediately follows from the CVS-approach that
  the (half relaxed) limits are between the minimal Ishii solution
  $\Um$ and the maximal one $\Up$. In the FL-approach, it is not clear
  what is the flux limiter of the solution of the approximating
  equation; it has to be identified before passing to the limit.

  We give two alternative proofs of the following
  result of \cite{IN} by combining the two approaches: the vanishing
  viscosity approximation converges towards the function $\Up$ defined
  in the CVS-approach. As in the proof of the comparison principle
  between flux-limited solutions, we are guided in the
    first proof of Theorem~\ref{teo:viscous} by the identification of
  specific slopes \cite{IM,IM-md}; see the introduction for more
  details and Lemma~\ref{linkHami} in the Appendix.
%-------------------------------------------------------------
\begin{theo}[The vanishing viscosity limit -- \cite{IN}]\label{teo:viscous}
    Assume [H0]-[H2].  \\
    For any $\eta>0$, let $u_\eta$ be the unique  
    solution in $L^\infty \cap W^{2,r}_{loc}$ (for any $r>1$) 
    of the following problem
    \begin{equation}\label{pb:viscous}
        -\eta  \Delta  u_\eta +u_\eta+  H(x, D u_\eta ) = 
        0\quad\text{in}\quad\R^N \,,
    \end{equation}
    where $H=H_1$ in $\Omega_1$ and $H=H_2$ in $\Omega_2$. 

    Then, as $\eta \to0$, the sequence $(u_\eta)_\eta$ 
    converges locally uniformly to $\Up$ in $\R^N$.
\end{theo}
%-------------------------------------------------------------------------
\begin{rem}\label{rem:reg}
  It is worth pointing out that, as long as $\eta>0$,
  it is not necessary to impose a condition on $\H$ because of the
  strong diffusion term. Moreover, the function $u_\eta$ is $C^1$
  since it is in $W^{2,r}_{loc}$ (for any $r>1$).
\end{rem}
%-------------------------------------------------------------------------
\begin{proof}
  We first recall that, by Theorem~\ref{thm:Um.Up}, $\Up$ is the
  maximal subsolution (and Ishii solution) of \eqref{eqn-minus} and we
  proved in Theorem~\ref{compBBC-IM} that it is the unique
  flux-limited solution of (HJ-FL) with $G=\HTreg$. We recall that
  \eqref{HJ1}-\eqref{HJ2} is completed in (HJ-FL) with the condition
\[
\max \Big( u(x)+\HTreg(x,D_\H u) , u(x)+ H_1^+(x, Du) , u(x)+
  H_2^-(x, Du) \Big) = 0 \: \quad \mbox{ on } \H
\]
 in the sense of Definition \ref{defiNEW}. Let us classically consider
 the half relaxed limits (see \cite{Ba} for a definition)
\[
 \underline{u}(x):={\rm liminf}_{ *}  \: u_\eta(x)  \quad \quad \overline{u}(x):={\rm limsup}^{*}   \:  u_\eta(x) \: . 
\]
We observe that we only need to prove the following inequality
\begin{equation} \label{equ:tesi} 
 \Up(x) \leq \underline{u}(x) \quad \mbox{ in } \R^N .
\end{equation}
Indeed, by the maximality of $\Up$ we have
$\overline{u}(x)\leq \Up(x)$ in $\R^N$; moreover, by construction we
have $\overline{u}(x)\geq \underline{u}(x)$ in $\R^N$, therefore if we
prove \eqref{equ:tesi} we can conclude that
$\Up(x) \leq \underline{u}(x) \leq \overline{u}(x)\leq \Up(x)$ which
implies that $(u_\eta)_\eta$ converges locally uniformly to $\Up$ in
$\R^N$.

Thanks to the arguments in \cite[Lemma 4.2]{BBC1} and \cite[Lemma
4.3]{BBC1} we can {\em regularize} and {\em localize} $\Up$. We can
then assume that $\Up$ is $C^1$ at least in the $x_1,\dots,x_{N-1}$
variables and that $\Up (x)- \underline{u}(x) \to -\infty$ as
$|x|\to +\infty$. For the sake of clarity, we continue to write $\Up$
for this subsolution.  Therefore, there exists $\bar{x} \in \R^N$ such
that
$$
 M:=\Up (\bar{x})- \underline{u}(\bar{x})=\sup_{x \in \R^N} \:  \big( \Up (x) -  \underline{u}(x) \big)  \:.
$$
We assume by contradiction that $M >0$.

We first remark that, necessarily, $\bar{x} \in \H$. Indeed,
otherwise, we can use classical comparison arguments for the $H_1$ or
$H_2$ equation, together with an easy localization argument, to get a
contradiction.

  Since $\Up$ is $C^1$ in the $x'$-variables,  the flux-limited subsolution condition  can be written as  
$$
 \Up(\bx)+\HTreg(\bx,D_{x'}\Up(\bx)) \leq 0  \:, 
$$
therefore by  the contradiction argument ($ \Up (\bar{x}) > \underline{u}(\bar{x}) $) we can suppose that 
$$
 -\Big(\frac{\Up(\bx)+\underline{u}(\bx)}{2} \Big)> \HTreg(\bx,D_{x'}\Up(\bx)) \: .
$$
By Lemma \ref{linkHami} in Appendix there exist two solutions
$\lambda_1,\lambda_2$, with $\lambda_2 < \lambda_1$, of the equation
$$
  \Htireg\Big(\bx,D_{x'}\Up(\bx)+\lambda e_N\Big)+ \frac{\Up(\bx)+
      \underline{u}(\bx)}{2}=0\,.
$$
Note that, since $\bar x$ and $p'=D_{x'}\Up(\bx)$ are fixed, $\lambda$
is a constant in the following construction of the test-function.  Let
$\chi(x_N,y_N)$ be defined as in \eqref{defchila} and
$$
  \psi_\eps(x,y):=\frac{|x'-y'|^2}{\eps^2}+\chi(x,y)+\frac{|x_N-y_N|^2}{\eps^2}
  + |x-\bx|^2\,.
$$
Note that  $\psi_\eps \in \T$ therefore, recalling that 
 $\underline{u}(\bx)={\rm liminf}_{ *}  \: u_\eta (\bx)$, we can consider the maximum points of 
$\Phi(x,y):=\Up(x)-\ueta(y)-\psi_\eps(x,y). $
More precisely, we set 
$$
  \Phi(x,y):=\max_{\R^N \times \R^N} ( \Up(x)-\ueta(y)-\psi_\eps(x,y) )\,.
 $$ 
 For the sake of simplicity of notation, we denote by $(x,y)$ a
 maximum point of $\Phi$ and we already notice that $x,y \to \bx$ as
 $\eps, \eta \to 0$.

We now consider 5 different  cases, depending on the position of $(x,y)$. 

\noindent\textbf{CASE 1/2: }$x_N>0$ and $y_N\leq0$ (or $x_N<0$ and $y_N\geq0$). We use
the subsolution condition for $\Up$ in $\Omega_1$ which gives
\[
H_1\Big(x,\frac{2(x'-y')}{\eps^2}+\lambda_1e_N+
\frac{2(x_N-y_N)}{\eps^2}e_N+o(1)\Big)+\Up(x)\leq0 \; .
\]
But, since $\Up$ is regular in the $x'$-variables, at a maximum point
of $\Phi$, we have (for some $o(1)$ due to the term $|x-\bx|^2$):
\begin{equation}\label{eq:grad.vfb}
    D_{x'}\Up(x)=2\frac{(x'-y')}{\eps^2}+o(1)\,.
\end{equation}
Therefore we can replace the $(x'-y')$-term by the 
gradient of $\Up$. Moreover, using that $H^-_1\leq H_1$, $H^-_1$ is non decreasing
and $(x_N-y_N)>0$ we get
\[
H^-_1\Big(x,D_{x'}\Up(x)+\lambda_1e_N
+o(1)\Big)\leq H_1\Big(x,D_{x'}\Up(x)+\lambda_1e_N+
\frac{2(x_N-y_N)}{\eps^2}e_N+o(1)\Big)\leq-\Up(x)\,.
\]
On the other hand, we recall that, by construction (see \cite{BBC1}),
the function $D_{x'}\Up$ is continuous, not only in $x'$ but also in
$x_N$. Therefore the regularity assumption on $H^-_1$ and the
construction of $\lambda_1$ yield
 $$
  H^-_1\Big(x,D_{x'}\Up(x)+\lambda_1e_N+o(1)\Big)= -\frac{\Up(\bx)+
      \underline{u}(\bx)}{2}+o(1)
 $$
 therefore,  since we assume
that $\Up(\bx)>\underline{u}(\bx)$, we obtain a contradiction for $\eps, \eta$ small enough. \\
The case $x_N<0$ and $y_N\geq0$ is completely similar,
using $H_2$ instead of $H_1$.  

\noindent\textbf{CASE 3/4:} $x_N=0$ and $y_N>0$ (or $<0$). We use the supersolution viscosity
inequality for $\ueta$ at $y$, replacing again the $(x'-y')$-term by $D_{x'}\Up$:
\begin{equation}\label{ineq.ud}
  -\frac{\eta C}{\eps^2}+
  H_1\Big(y,D_{x'}\Up(x)+\lambda_1e_N+
  \frac{2(x_N-y_N)}{\eps^2}+o(1)\Big)+\ueta(y)\geq0\,.
\end{equation}
We first want to show that we can replace $H_1$ by $H_1^+$ in this
inequality.  Indeed, using successively that $H_1^-(y,\cdot)$ is
nondecreasing (in the $p_N$-variable), the continuity of $D_{x'}\Up$,
the fact that $x_N-y_N=-y_N<0$, the definition of $\lambda_1$, the
regularity of $H_1^-$ and the contradiction assumption, we have
\begin{align*}
 -\frac{\eta C}{\eps^2}+
  H^-_1\Big(y,D_{x'}\Up(x)+\lambda_1e_N+
  \frac{2(x_N-y_N)}{\eps^2}+o(1)\Big)+\ueta(y) \\
\leq  -\frac{\eta C}{\eps^2}+ H^-_1\Big(\bx,D_{x'}\Up(\bx)+\lambda_1e_N \Big)+\ueta(y)+o(1)\\
\leq   -\frac{\eta C}{\eps^2}- \frac{\Up(\bx)+
      \underline{u}(\bx)}{2} +\ueta(y)+o(1) < 0
\end{align*}
for $\eta,\eps$ and $\frac{\eta}{\eps^2}$ small enough. We deduce that
\eqref{ineq.ud} holds true with $H^+_1$. \\
Moreover, by the subsolution condition  of $\Up$ on $\H$ we have 
$$
 H^+_1\Big(x,D_{x'}\Up(x)+\lambda_1e_N+ \frac{2(x_N-y_N)}{\eps^2}+o(1)\Big)+\Up(x)\leq0\,
$$
therefore the conclusion follows by standard arguments  putting together the two
inequalities for $ H^+_1$ and letting first $\eta$ and then 
$\eps$ tend to zero. If $y_N< 0$, we can repeat the same argument using $H^-_2$.

\noindent \textbf{CASE 5:} $x_N=y_N=0$.  Let us remark that this case
is not possible. We observe that $\ueta$ is regular (see
Remark~\ref{rem:reg}) therefore if we have a minimum point of
$x \mapsto \ueta-(\Up-\psi_\eps(x,y) )$, by construction of the
function $\chi$ we have $\lambda_1\geq \lambda_2$.  Since by
definition (Lemma \ref{linkHami} below) we have
$\lambda_2 < \lambda_1$ we obtain a contradiction.

%%%%%%%%%%%%%%FINE MANU%%%%%%%%%%%%%%%%%%%%%%

\end{proof}

\section{On the Kirchoff condition}

The Kirchoff condition is used in \cite{IM,IM-md} in order to pass to
the limit in the vanishing viscosity method. The connection between
the Kirchoff condition and a flux-limited solution is made
afterwards. In this section, we show that the Kirchoff condition leads
to the $\Up$-solution.  This Kirchoff condition is not easy to express
in our context since we would have to write
$$ -\frac{\partial u}{\partial x_N}-\frac{\partial u}{\partial
  (-x_N)}= 0\quad\hbox{on }\H\;,$$
but of course this has to be understood with test-functions in $\T$,
which are not $C^1$ in the normal variable across the interface.  The
precise definition on $\H$ is the following
%----------------------------------------------------------------------------
\begin{defi}[Solutions for the Kirchoff condition]
  \label{defiKC} An upper semi-continuous (usc), bounded function
  $u: \R^N \rightarrow \R$ is a \emph{subsolution for the Kirchoff
    Condition on $\H$} if for any test-function $\psi \in \T$ and any
  local maximum point $x \in \H$ of $x \mapsto (u-\psi)(x)$ in $\R^N$,
  we have
\begin{equation}
  \min \Big( -\frac{\partial \psi_1}{\partial x_N}+\frac{\partial \psi_2}{\partial x_N} ,  u(x)+ H_1(x, D\psi_1) ,  
  u(x)+ H_2(x, D\psi_2) \Big) \leq 0  \: .
\end{equation}
We say that a lower semi-continuous (lsc), bounded function
$v: \R^N \rightarrow \R$ is an \emph{supersolution for the Kirchoff Condition on $\H$} if for any function $\psi \in \T$ and any local
mininum point $x \in \H$ of $x \mapsto (v-\psi)(x)$ in $\R^N$, we have
\begin{equation}\label{cond.super.kirchoff}
  \max \Big(-\frac{\partial \psi_1}{\partial x_N}+\frac{\partial \psi_2}{\partial x_N} ,   v(x)+ H_1(x, D\psi_1) ,  
  v(x)+ H_2(x, D\psi_2) \Big) \geq 0  \: .
\end{equation}
\end{defi}
%-------------------------------------------------------------------------
\begin{rem}
In \cite{IM,IM-md,IN}, an equivalent notion of solutions
  is introduced for general (and generalized) junction
  conditions. They are referred to as \emph{relaxed solutions}.
\end{rem} 
The following result describes the link with
  flux-limited solutions.  In particular, the proposition below
  implies that solutions for the Kirchoff conditions are unique. It
  also implies that the vanishing viscosity limit selects $\Up$
  (Theorem~\ref{teo:viscous}).
%-------------------------------------------------------------------------
\begin{pro}\label{pro}Assume [H0]-[H2]. \\
$(i)$ If $u$ is a subsolution for the Kirchoff Condition then $u$ is a
flux-limited subsolution with $\HTreg$.\\
$(ii)$  If $v$ is a supersolution for the Kirchoff Condition then $v$ is a
flux-limited supersolution with~$\HTreg$.
\end{pro}
%------------------------------------------------------------------------
\begin{proof}
To prove $(i)$, we first notice that subsolutions for the Kirchoff Condition are
Lipschitz continuous; to prove it, we just modify the classical proof in the
following way: for $0<\kappa \ll 1$ and $x\in \R^N$, we consider the maximum points of the function
$$ y\mapsto u(y)-C|y-x| -\kappa\exp(-2y_N^+ - y_N^-)\; ,$$
the new, ``small'' term $\kappa\exp(-2y_N^+ - y_N^-)$ being there to avoid that the inequality 
$$-\frac{\partial \psi_1}{\partial x_N}+\frac{\partial \psi_2}{\partial x_N}\leq 0\;$$
holds. Using this remark, the coercivity of $H_1,H_2$ and a large enough $C$,
allows to conclude that, for any $y$ (and $x$) $$u(y)-C|y-x| -\kappa\exp(-2y_N^+
- y_N^-) \leq u(x)\; ,$$
which proves the Lipschitz continuity by letting $\kappa$ tend to $0$.

Next we use the following lemma which is a direct consequence of \cite[Lemma~5.3]{Ba}.

\begin{lem}\label{JAB}Assume [H0]-[H2]. If $u$ is a Lipschitz continuous subsolution of 
\[ \left\{ \begin{array}{l}
u+H_1(x,Du)=0\quad\hbox{ in }\Omega_1\; , \\
 u+H_2(x,Du)=0\quad\hbox{ in }\Omega_2\; ,
\end{array} \right.
\]
then it is a subsolution of
\( \max(u+H^+_1(x,Du),u+H^-_2(x,Du))=0\) on $\H$.
\end{lem}
In view of Lemma~\ref{JAB}, it is enough to show that
\[
	u(x)+\HTreg\big(x,D_{\H}\psi (x',0)\big)\leq0\;,
\]
at any strict local maximum point $x=(x',0)$ of $ y \mapsto u(y)-\psi(y)$ in
$\R^N$ where  $\psi\in \T$.

In particular, $x'$ is a strict local maximum point of $y'\mapsto
u(y',0)-\psi(y',0)$ on $\H$ and we consider the function 
\begin{equation}
    \label{test}
y=(y',y_N) \mapsto u(y)-\psi(y',0)-\chi(y_N)-\frac{(y_N)^2}{\eps^2}\; ,
\end{equation}
with, for some small $\kappa >0$
$$ \chi(y_N):= \left\{\begin{array}{ll}
 (\lambda -\kappa)y_N & \mbox{ if }  y_N \geq 0 ,\\
 (\lambda +\kappa)y_N & \mbox{ if }  y_N < 0 , \\
 \end{array}
 \right.
 $$
 where $\lambda$ is given by Lemma \ref{lem:H1m.H2p} as follows: let
 $(x,p'):=(x,D_{\H}\psi(x',0))$ we choose $\lambda=s^*$ in the three cases 1,
 2 and 3.  Note that this is, roughly speaking, the minimal
 intersection point between $H^-_1$ and $H^+_2$ and therefore we have
 \begin{equation}\label{condlamb}
  \HTreg (x,D_{\H}\psi(x',0))=H^-_1(x,D_{\H}\psi(x',0)+\lambda e_N)=
  H^+_2(x,D_{\H}\psi(x',0)+\lambda e_N) \: .
 \end{equation}
  
 By standard arguments, the function defined in \eqref{test} has a
 maximum point $z=(z',z_N)$ near~$x$. Of course, $z$ depends on $\eps$
 but we drop this dependence for the sake of simplicity. Since $x'$ is a strict local maximum point of $y'\mapsto
u(y',0)-\psi(y',0)$ on $\H$, it is clear that $z\to x$ as $\eps \to 0$.

 The first case we examine is when $z_N=0$, where necessarily
 $z=x$. By the definition of subsolution for the Kirchoff condition,
 we have
\[
\min(-(\lambda -\kappa)+(\lambda +\kappa),
u(x)+H_1(x,D_{\H}\psi(x',0)+(\lambda-\kappa)e_N),
u(x)+H_2(x,D_{\H}\psi(x',0)+(\lambda+\kappa)e_N))\leq 0.
\]
But $-(\lambda -\kappa)+(\lambda +\kappa)=2\kappa >0$, therefore
\begin{equation} \label{caso1}
\min(u(x)+H_1(x,D_{\H}\psi(x',0)+(\lambda-\kappa)e_N),
u(x)+H_2(x,D_{\H}\psi(x',0)+(\lambda+\kappa)e_N))\leq 0.
\end{equation}
Letting $\kappa \to 0$ yields the desired inequality thanks to \eqref{condlamb} since $H_1 \geq H^-_1$ and $H_2 \geq H^+_2$. 

 If $z_N>0$,  by  the subsolution condition in $\overline{\Omega_1}$ we have
\begin{equation} \label{caso2}
 H_1\Big(z,D_{\H}\psi(z',0)+(\lambda-\kappa)e_N+ \frac{2z_N}{\eps^2}\Big)
+u(z)\leq 0\; ,
\end{equation}
while if $z_N < 0$ we obtain 
\begin{equation} \label{caso3} 
H_2\Big(z,D_{\H}\psi(z',0)+(\lambda+\kappa)e_N+ \frac{2z_N}{\eps^2}\Big)
+u(z)\leq 0\; . 
\end{equation}
We claim now that the conclusion follows from \eqref{condlamb} with similar arguments in these two cases. For instance  if 
\eqref{caso2} holds   according to Lemma~\ref{linkHami} and using the fact that $H^-_1$ is nondecreasing
\begin{eqnarray*}
H_1\Big(z,D_{\H}\psi(z',0)+(\lambda-\kappa)e_N+ \frac{2z_N}{\eps^2}\Big) 
& \geq &H^-_1\Big(z,D_{\H}\psi(z',0)+(\lambda-\kappa)e_N+ \frac{2z_N}{\eps^2}\Big)\\
& \geq & H^-_1(z,D_{\H}\psi(z',0)+(\lambda-\kappa)e_N) \\
& = &\HTreg (x,D_{\H}\psi(x',0)) + o_\eps(1) + o_\kappa(1)\; .
\end{eqnarray*}
Therefore
$$ \HTreg (x,D_{\H}\psi(x',0)) + o_\eps(1) + o_\kappa(1)+u(z)\leq 0\; .$$
And the conclusion follows by letting first $\eps$ tend to $0$ and then
$\kappa$ tend to $0$.  Of course, an analogous computation is valid
for $H_1$ even if $z_N=0$ or for $H_2$ if $z_N\leq 0$
and the proof of $(i)$ is complete in cases \eqref{caso2} and \eqref{caso3}.

We now turn to the proof of $(ii)$. Consider a test function $\psi\in\T$ such 
that $v-\psi$ reaches a local strict minimum at $x=(x',0)$. We are going to prove
that for all $\eps >0$, 
\begin{equation}\label{desired}
\max( v (x) + \HTreg (x,p')+\eps, v(x) + H_1^+ (x,p' + p_1 e_N), 
v(x) + H_2^- (x,p'+p_2e_N)) \ge 0
\end{equation}
where $p' = D_{\H} \psi (x)$ and $\displaystyle p_i = \frac{\partial \psi_i}{\partial x_N}(x)$. 

It is convenient to write $\bar A = - v(x)$ and $A^\eps = \HTreg (x,p') + \eps$. 
We argue by contradiction by assuming that \eqref{desired} does not hold true,
which means
\begin{equation}\label{absurd}
 A^\eps < \bar A, \quad H_1^+(x,p' + p_1e_N) < \bar A, \quad H_2^- (x,p'+p_2e_N) < \bar A.
\end{equation}

Since $A^\eps>\HTreg(x,p')$ we can find $\lambda_1^\eps > \lambda_2^\eps$ such
that (see Appendix) 
\[A^\eps = H_1^-(x,p'+\lambda_1^\eps e_N) = H_2^+(x,p'
+\lambda_2^\eps e_N) = H_1(x,p'+\lambda_1^\eps e_N) = H_2(x,p'
+\lambda_2^\eps e_N)\,.
\] 

Now we use the notion of critical slopes introduced 
in \cite[Lemma~2.8]{IM-md}: we set 
$$ \kappa_1:=\liminf_{\genfrac{}{}{0pt}{}{y\to x}{y_N>0}} \frac{(v(y)-\psi(y))-(v(x)-\psi(x))}{y_N}\quad \hbox{and} \quad \kappa_2:=\liminf_{\genfrac{}{}{0pt}{}{y\to x}{y_N<0}} \frac{(v(y)-\psi(y))-(v(x)-\psi(x))}{-y_N}\; .$$
By definition, $\kappa_1,\kappa_2 \geq 0$ can be infinite and, for any $q_1\leq \kappa_1$ and $q_2 \leq \kappa_2$, there exists a function $\phi =(\phi_1,\phi_2) \in \T$ such that, for $i=1,2$, $D_{\H} \phi_i (x) =0$ and $\frac{\partial \phi_i}{\partial x_N}(x)=q_i$ and the function $y\mapsto v(y)-\psi(y)-\phi(y)$
has a strict local minimum point at $x$.

The proof of this claim is analogous to the proof of the equivalence of the two classical definitions of viscosity supersolutions by subdifferential and by testing with smooth functions : if $\chi:\R \to \R$ is defined by
$$ \chi(s) =\begin{cases} q_1 s & \text{if }s\geq 0\,,\\
        q_2 s & \text{if }s\leq 0\,,\end{cases}
$$
then, by the definition of $\kappa_1,\kappa_2$, we have 
$$ (v(y)-\psi(y))-(v(x)-\psi(x)) \geq \chi(y_N) + |y_N|o(1)=\chi(y_N) + o(|y-x|)\; ,$$
and the proof consists in regularizing the $o(|y-x|)$ in a suitable way.

If these suprema are finite (otherwise the following claim just follows from the coercivity properties for $H_1,H_2$ by taking $\kappa_1,\kappa_2$ large enough), we claim that
\begin{equation}\label{critical}
v (x) + H_1 (x,p'+(p_1+\kappa_1)e_N ) \ge 0 \quad \text{ and } \quad
v(x) + H_2 (x,p'+(p_2 -\kappa_2)e_N) \ge 0\; .
\end{equation}
Indeed these properties are obtained by looking at 
$y\mapsto
v(y)-\psi(y)- \phi (y)-\eta y_N^+$ and $y\mapsto
v(y)-\psi(y)- \phi (y)-\eta y_N^-$ for $\eta$ small enough where $\phi \in \T$ is the function defined as above but for $q_1=\kappa_1$ and $q_2=\kappa_2$.

By definition of the critical slopes, the maximum is necessarily achieved in 
$\Omega_1$ in the first case and in $\Omega_2$ for the second one, otherwise the minimum property would lead to a contradiction to the liminf definition of $\kappa_1,\kappa_2$. Letting
$\eta$ tend to $0$ in the viscosity inequalities yields the claim.

Then we can write \eqref{critical} as
\( H_1 (x,p'+(p_1+\kappa_1)e_N) \ge \bar A \)  and \( H_2 (x,p'+(p_2
    -\kappa_2)e_N) \ge \bar A\).
Since $\kappa_1\geq 0$ and $H_1^+$ is non-increasing in the $e_N$-direction, by
\eqref{absurd} we get
$$H_1^+(x,p'+(p_1+\kappa_1)e_N)\leq H_1^+(x,p' + p_1e_N) < \bar A\,.$$
Therefore,
necessarily $H_1 (x,p'+(p_1+\kappa_1)e_N)=H_1^-(x,p'+(p_1+\kappa_1)e_N)\geq \bar
A$ and in the same way, $ H_2 (x,p'+(p_2 -\kappa_2)e_N)
=H_2^+(x,p'+(p_2-\kappa_2)e_N)\geq
\bar A$.

Using an analogous monotonicity argument, 
$H_1^-(x,p'+(p_1+\kappa_1)e_N)\geq\bar A>A^\eps$ implies that
$p_1 + \kappa_1 > \lambda_1^\eps$ and, in the same way,
$p_2 - \kappa_2 < \lambda_2^\eps$. Therefore $q_1= \lambda_1^\eps-p_1 < \kappa_1$, $q_2 = p_2 - \lambda_2^\eps < \kappa_2$ and if $\phi\in \T$ is the function defined as above with $q_1$ and $q_2$, the function
 $ y\mapsto v(y)-\psi(y)-\phi (y)$
reaches a minimum at $x=(x',0)$. We can use $\psi-\phi \in\T$ as a test-function for $v$ in 
the Kirchoff condition \eqref{cond.super.kirchoff}. 
From $\lambda_1^\eps>\lambda_2^\eps$, it follows that at $x$, the first term
gives a negative contribution
$$-\frac{\partial(\psi_1-\phi_1)}{\partial x_N}
+\frac{\partial (\psi_2-\phi_2)}{\partial x_N}=-\lambda_1^\eps + \lambda_2^\eps
<0\,.$$
Hence, the supersolution condition reduces to \[ \max (v (x) + H_1
    (x,p'+\lambda_1^\eps e_N), v(x) + H_2 (x,p'+\lambda_2^\eps e_N)) \ge 0 \]
which means $v(x) + \HTreg(x,p') + \eps \ge 0$ by the definition of
$\lambda_1^\eps,\lambda_2^\eps$. But then we reach a contradiction with 
$A^\eps<\bar A$.  Then, $(ii)$ follows from letting $\eps$ tend to zero in
\eqref{desired}.
\end{proof}

\appendix

\section{Appendix}

In this appendix, we decompose any vector $p\in\R^N$ as $p=(p',p_N)$, but also
as $p=p'+p_Ne_N$ (with a slight abuse of notation). We will concentrate here
only on $H_1^-$ and $H_2^+$, defined respectively 
by \eqref{def:Ham1-} and \eqref{def:Ham2+}. 

Notice first that for any fixed $(x,p')$, the functions
$s\mapsto H_1(x,p'+se_N)$ and $s\mapsto H_2(x,p'+se_N)$ are convex and
coercive, hence each of them reaches its minimum. We introduce the following
notation:
$$\begin{aligned}
    \Honemin(x,p')& :=\min_{s\in\R}H_1(x,p'+se_N)\,,\\
    \Htwomin(x,p')& :=\min_{s\in\R}H_2(x,p'+se_N)\,.
\end{aligned}$$
Since the minimum can possibly be attained on a whole interval, we
set
$$
\begin{aligned}
    m_1(x,p'):= & \sup\big\{s\in\R:H_1(x,p'+s e_N)=\Honemin(x,p')\big\}\,,\\
    m_2(x,p'):= & \inf\big\{s\in\R:H_2(x,p'+s e_N)=\Htwomin(x,p')\big\}\,,
\end{aligned}
$$
and in the following for $\Honemin,\Htwomin,m_1,m_2$ we skip the reference to
$(x,p')$ since this pair of variable is always fixed.
%------------------------------------------------------------------------------
\begin{lem}\label{lem:H1m.H2p}
  Assume [H0]-[H2] and [HG]. Then the Hamiltonians $H^-_1$ and $H^+_2$
  satisfy
    $$
    H_1^-(x,p)=\begin{cases} \Honemin & \text{if }p_N\leq m_1\,,\\
        H_1(x,p) & \text{if }p_N>m_1\,,\end{cases}
    \qquad
    H_2^+(x,p)=\begin{cases} H_2(x,p) & \text{if }p_N\leq m_2\,,\\
        \Htwomin & \text{if }p_N>m_2\,.\end{cases}
    $$
    As a consequence, $H^-_1(x,p)$ is nondecreasing in the $p_N$-variable, and
    $H_2^+$ is nonincreasing in the $p_N$-variable. Moreover $H^-_1(x,p)$ is strictly increasing in the $p_N$-variable for $p_N>m_1$ and
    $H_2^+$ is strictly decreasing in the $p_N$-variable for $p_N<m_2$\end{lem}
%------------------------------------------------------------------------------

Figure \ref{fig:ham} illustrates a typical situation where $H_1$ has a
flat portion at its min, while $H_2$ is strictly convex. Here, $(x,p')$ is fixed
and $s$ is the variable.

\begin{figure}[!ht]
\begin{center}
    \includegraphics[width=15cm]{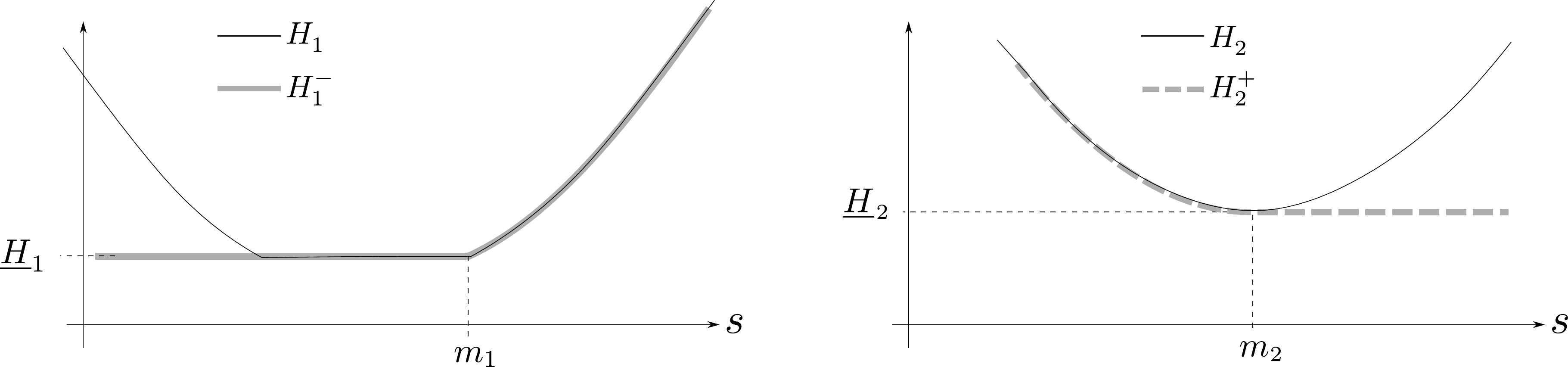}
\end{center}
\caption{Typical situation.}
\label{fig:ham}
\end{figure}

\begin{proof}
 We provide the proof for $H_1$ only, since it is the same for $H_2$. Notice first
    that obviously, by definition $H_1=\max(H_1^-;H_1^+)$. 

    Next, the minimum of the convex, coercive function
    $s\mapsto H_1(x,p'+se_N)$ is achieved at some $\bar s \in \R$ and
    then standard results of convex analysis show that the maximum
    which defines $H_1(x,p'+\bar se_N)$ is attained for a control
    $\alpha_*\in A_1$ such that $b_1(x,\alpha_*)\cdot e_N=0$. Hence we
    can use this specific control in the supremum for $H_1^-(x,p)$ and
    we deduce that $H_1^-(x,p)\geq \Honemin$. A small modification of
    this argument shows also that $H_1^+(x,p)\geq\Honemin$ (we need to
    add a little bit of controlability here because the supremum for
    $H_1^+$ requires $b\cdot e_N>0$, not $b\cdot e_N=0$).
    
    Then, we have $H_1^-(x,p)\leq \Honemin$ if $p_N\leq m_1$
    since $s\mapsto H^-_1(x,p'+se_N)$ is increasing, and a similar
    argument shows that for $p_N\geq m_1$,
    $H^+_1(x,p)\leq \Honemin$. Hence we deduce that
    $$H_1(x,p)=\begin{cases}
        H_1^+(x,p) & \text{if }p_N\leq m_1\,,\\
        H_1^-(x,p) & \text{if }p_N>m_1\,.
    \end{cases}$$
 For $p_N >m_1$,  the convex function $p_N \mapsto H_1(x,p'+p_N e_N)$ cannot have $0$ in its subdifferential (otherwise at such a point we would have a minimum point, which would contradict the definition of $m_1$) and therefore by the classical Mean Value Theorem for convex functions in $1-d$, this function is increasing for $p_N >m_1$.
\end{proof}

For any $x \in \R^N$, $p \in \R^N$ we define the Hamiltonians 
\begin{equation}  \label{defHti}
\Hti(x,p):= \max( H_1(x,p), H_2(x,p) )
\end{equation}
\begin{equation}   \label{defHtireg} 
\Htireg(x,p):= \max(H_{1}^- (x,p),H_{2}^+ (x,p))\; .
\end{equation}
\begin{rem} \label{lemCOE} We notice that the Hamiltonian $\Hti$ is
  convex and coercive in the $p$-variable (since it is  the
  maximum of two convex and coercive Hamiltonians). Moreover the same
  properties hold for $\Htireg$ thanks to the structure of $H_{1}^-$
  and $H_{2}^+$ proved in Lemma \ref{lem:H1m.H2p}.
\end{rem}
We recall that the Hamiltonians $\Hti$ and $\Htireg$ are convex and coercive in the $p$-variable (since we are taking the maximum of two convex
Hamiltonians). Moreover, we have
    \begin{equation} \label{HT1}
    \HT(x,p^\prime)= \min_{s\in\R} \Hti(x,p^\prime+s e_N)\,,
    \end{equation}
    \begin{equation} \label{HT2}
        \HTreg(x,p^\prime)= \min_{s\in\R} \Htireg(x,p'+s e_N)\,.
    \end{equation}
Indeed, equalities \eqref{HT1}-\eqref{HT2} follow from the definition of
$\HT$ and $\HTreg$ (see Remark \ref{remHpm}) and classical results in
convex analysis. For a detailed similar argument see the proof of
Theorem 3.3 case 1 in \cite{BBC1}.

The next step consists in introducing the function
\begin{equation}   \label{def-Phi} 
\phi(\lambda):=\Htireg(x,p'+\lambda e_N)\,.
\end{equation}
We are going to describe the different types of situations for this function $\phi$ and the consequences for the values of $ \HT(x,p^\prime)$,   $\HTreg(x,p^\prime)$ and for the equations
\begin{equation}   \label{def-EqnProof} 
H^-_1(x,p^\prime+ \lambda_1 e_N)=A \quad \mbox{ and } \quad
H^+_2(x,p^\prime+ \lambda_2 e_N)=A,
\end{equation} 
which appear in the proof of Theorem~\ref{comp-IM}. To do so, we introduce the functions 
$f_1(s):=H_1^-(x,p'+s e_N)$ and $f_2(s):= H_2^+(x,p'+s e_N)$. Since $f_1(s) \to
+\infty$ as $s\to +\infty$ and remains bounded as $s\to -\infty$, while $f_2(s)
\to +\infty$ as $s\to -\infty$ and remains bounded as $s\to +\infty$  (see
Figure \ref{fig:ham}), there exists at least a solution of the equation
$f_1(s)=f_2(s)$ and we denote by $s_*$ the minimal solution. 
By the monotonicity
properties of $f_1$ and $f_2$, it follows that $f_2>f_1$ for $s <s^*$ while
$f_2\leq f_1$ for $s\geq s^*$. Taking into account the flat portions of $H_1^-$
and $H_2^+$ where they reach their respective minimum, 
we arrive at the following complete description.
%--------------------------------------------------------------------------------
\begin{lem} \label{linkHami}\ \\[2mm]
    \noindent $(i)$  There are three possible configurations.
    \begin{itemize}
        \item[ ] {\bf Case 1 :} $s_* \leq  m_1$ and $s_* \leq m_2$ where (see
            Fig. \ref{fig_2}) 
    \begin{equation}\label{eq:phi.lambda1}
      \phi(\lambda)=
      \begin{cases}
          H_2^+(x,p'+\lambda e_N) & \text{if }\lambda<s_*\,,\\
          \Honemin (x,p') & \text{if }\lambda\in[s_*,m_1]\,,\\
          H_1^-(x,p'+\lambda e_N) & \text{if }\lambda>m_1\,.\\
      \end{cases}
    \end{equation}
\item[ ]{\bf Case 2 :} $s_* > m_1$ and $s_* \geq m_2$ where
    \begin{equation}\label{eq:phi.lambda2}
      \phi(\lambda)=
      \begin{cases}
          H_2^+(x,p'+\lambda e_N) & \text{if }\lambda\leq m_2\,,\\
           \Htwomin (x,p')
          & \text{if }\lambda\in[m_2,s_*]\,,  \\
          H_1^-(x,p'+\lambda e_N) & \text{if }\lambda\geq s_*\,.\\
      \end{cases}
    \end{equation}
\item[ ]{\bf Case 3 :} $s_* > m_1$ and $s_* \leq m_2$ where (see Fig.
    \ref{fig_1})
    \begin{equation}\label{eq:phi.lambda3}
      \phi(\lambda)=
      \begin{cases}
          H_2^+(x,p'+\lambda e_N) & \text{if }\lambda<s_*\,,\\
          H_1^-(x,p'+\lambda e_N) & \text{if }\lambda>s_*\,.\\
      \end{cases}
    \end{equation}
\end{itemize}

\noindent $(ii)$ In Cases 1 \& 2, we have
$$ \HTreg(x,p^\prime)= \max(\Honemin (x,p'),\Htwomin (x,p'))\; ,$$
while, in Case 3, $ \HTreg(x,p^\prime)=\HT(x,p^\prime)$.

\noindent $(iii)$ Finally, for any
$A > \max(\Honemin (x,p'),\Htwomin (x,p'))$ there exist a unique pair
$\lambda_2<\lambda_1$ such that
$$
H^-_1(x,p^\prime+ \lambda_1 e_N)=A \quad \mbox{ and } \quad
H^+_2(x,p^\prime+ \lambda_2 e_N)=A
$$
and the same equations hold with $H_1$ and $H_2$ instead of $H^-_1$
and $H_2^+$.
\end{lem}
%-------------------------------------------------------------------------------
\begin{figure}[!ht]
\begin{center}
 \includegraphics[width=10cm]{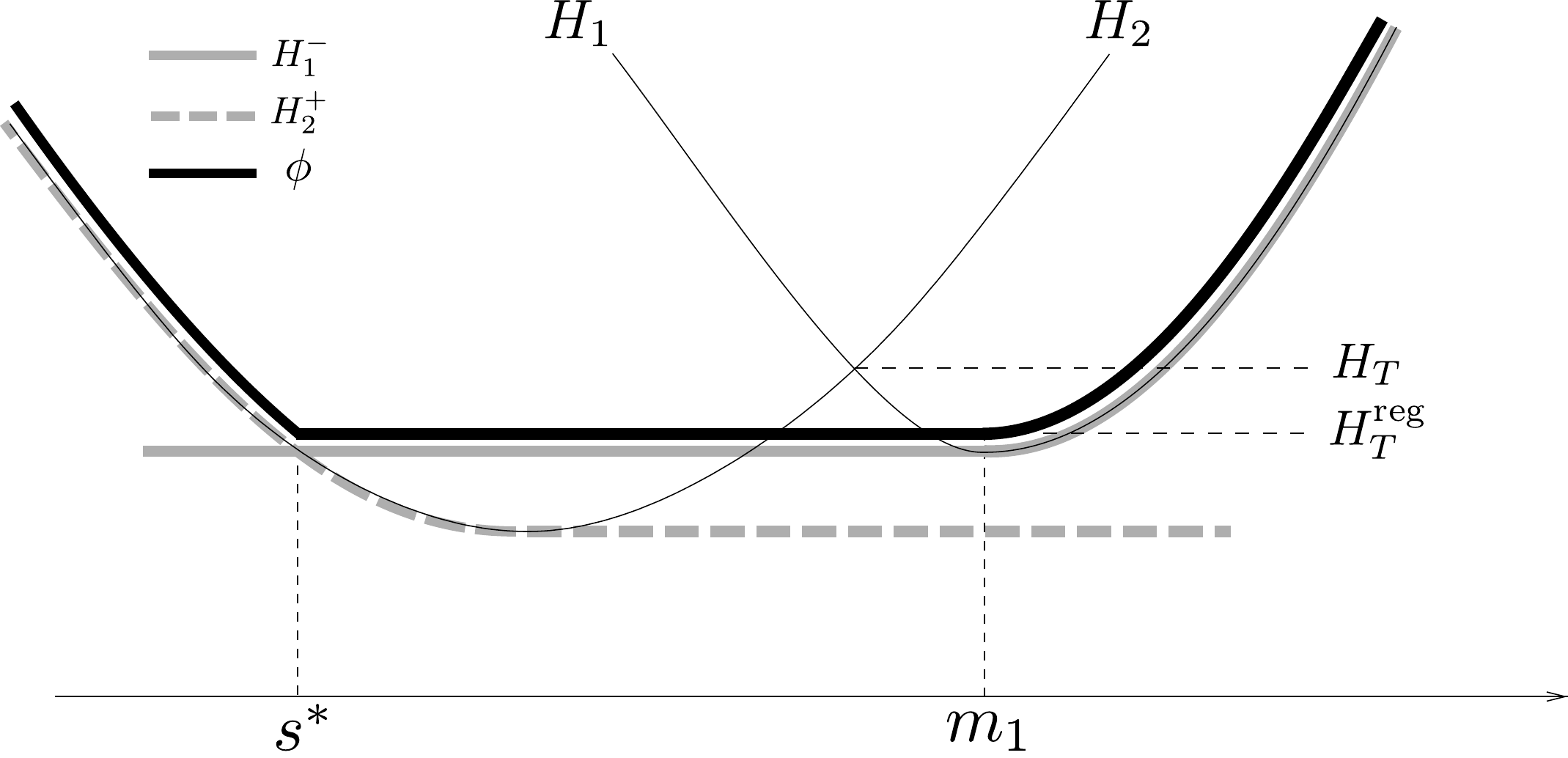}
\end{center}
\caption{$\HT(x,p^\prime)> \HTreg(x,p^\prime)=\Htwomin(x,p')$.}
\label{fig_2}
\end{figure}

\begin{figure}[!h]
\begin{center}
    \includegraphics[width=10cm]{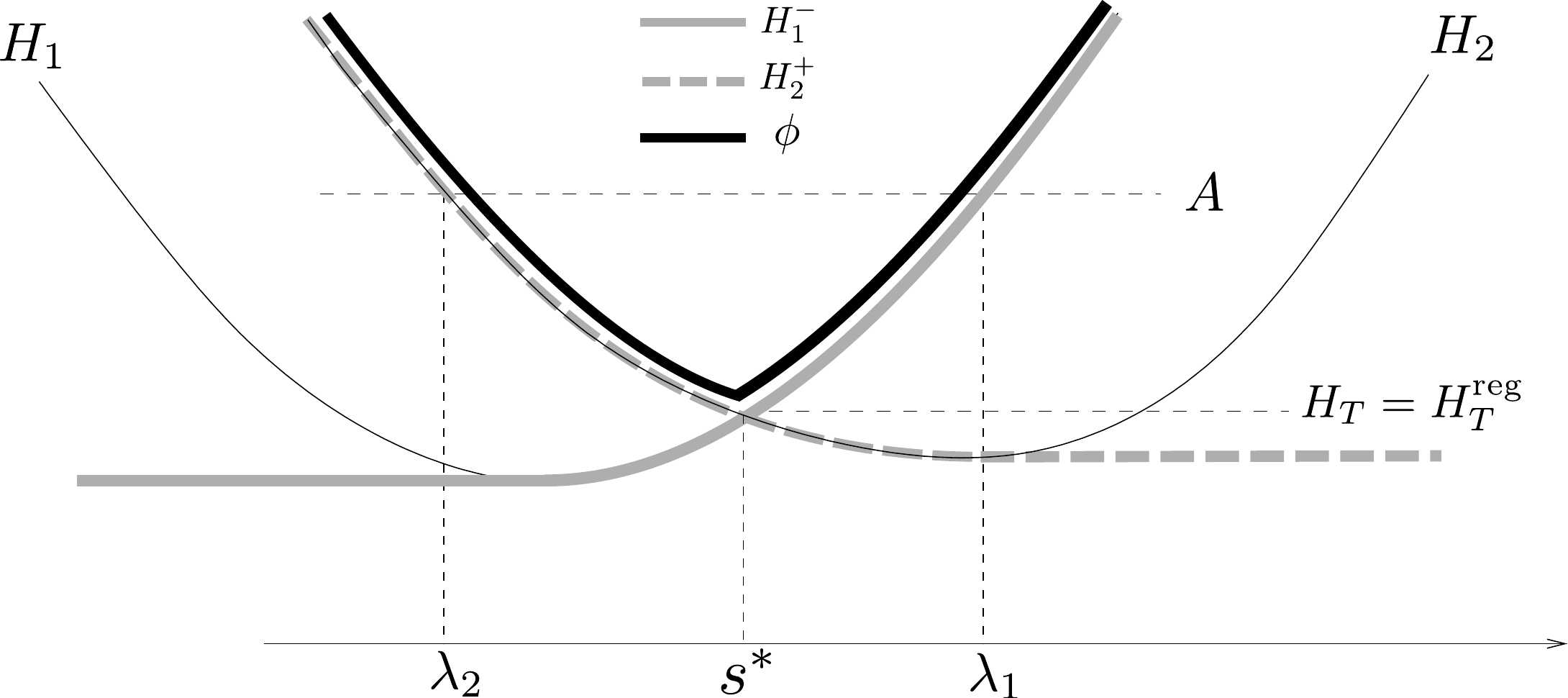}
\end{center}
\caption{$\HT(x,p^\prime)=\HTreg(x,p^\prime)$.}\label{fig_1}
\end{figure}

\bibliographystyle{plain}
\bibliography{bbcim}

\end{document}